\documentclass[12pt,letterpaper]{article}

\usepackage{geometry} 
\usepackage{microtype}
\usepackage{mathtools}
\usepackage{color,latexsym,amssymb,amsthm,graphicx,subfigure}
\usepackage{cite}

\usepackage[small]{titlesec}

\usepackage{fullpage}

\let\eps\varepsilon

\newcommand{\CC}{\mathbb{C}}
\newcommand{\RR}{\mathbb{R}}

\newcommand{\pts}{\mathcal{P}}
\newcommand{\circles}{\mathcal{C}}
\newcommand{\spheres}{\mathcal{S}}

\newtheorem{thm}{Theorem}[section]
\newtheorem{lem}{Lemma}[section]
\newtheorem{cor}{Corollary}[section]
\newtheorem{example}{Example}[section]

\newtheorem*{circlePartitioningLemLem}{Lemma \ref{circlePartitioningLem}}

\theoremstyle{remark}
\newtheorem{rem}{Remark}[section]
\newtheorem{defn}{Definition}[section]

\begin{document}

\title{Breaking the 3/2 barrier for unit distances in three dimensions}
\author{Joshua Zahl\thanks{University of British Columbia, Vancouver BC}}
\maketitle

\begin{abstract}
We prove that every set of $n$ points in $\RR^3$ spans $O(n^{295/197+\eps})$ unit distances. This is an improvement over the previous bound of $O(n^{3/2})$, which was a natural barrier for this problem. A key ingredient in the proof is a new result for cutting circles in $\RR^3$ into pseudo-segments.
\end{abstract}

\maketitle

\section{Introduction}
This paper proves the following theorem.
\begin{thm}\label{mainThm}
Every set of $n$ points in $\RR^3$ spans  $O(n^{\frac{295}{197}+\eps})=O(n^{\frac{3}{2}-\frac{1}{394} +\eps})$ unit distances.
\end{thm}
This is a small improvement over the previous bound of $O(n^{3/2})$, which was proved independently by Kaplan, Matou\u{s}ek, Pat\'akov\'a, and Sharir in \cite{KMSS} and by the author in \cite{Z}. However, it is still far from the conjectured optimal bound of $O(n^{4/3})$.

To put Theorem \ref{mainThm} in context, we will give a brief history of incidence geometry in Euclidean space. In \cite{KST}, K\H{o}v\'ari, S\'os, and Tur\'an showed that if $G$ is a bipartite graph with edge sets of size $m$ and $n$ that does not contain an induced copy of $K_{s,t}$, then $G$ has at most $t^{1/s}mn^{1-1/s}+sn$ edges. This theorem can be used to prove many results in incidence geometry. For example, since every pair of distinct points uniquely determines a line, there are $O(n^{3/2})$ incidences between $n$ points and $n$ lines in the plane. Similarly, since at most two unit spheres can pass through any three points in $\RR^3$, there are $O(n^{5/3})$ unit distances spanned by $n$ points in $\RR^3$.

However, the incidence theorems given by the K\H{o}v\'ari-S\'os-Tur\'an theorem are frequently not sharp. For example, Szemer\'edi and Trotter proved in \cite{SzTr} that $n$ points and $n$ lines in the plane can have at most $O(n^{4/3})$ incidences, and this is sharp. To do this, they employed a technique now known as ``partitioning + K\H{o}v\'ari-S\'os-Tur\'an.'' In short, they decomposed the plane into a union of open connected sets (called ``cells''), plus a ``boundary.'' Each point in the plane lies in at most one of these cells. Each line can intersect several of these cells, but the number of cells that each line can intersect is controlled. Szemer\'edi and Trotter then examined the collection of points and lines inside each cell, applied the K\H{o}v\'ari-S\'os-Tur\'an theorem, and summed the resulting contribution over all cells in the partition.

In \cite{CEGSW}, Clarkson, Edelsbrunner, Guibas, Sharir, and Welzl systematically extended this technique to prove incidence theorems in the plane and in higher dimensions. Amongst many other results, they proved that $n$ points in $\RR^3$ span $O(n^{3/2}\beta(n))$ unit distances, where $\beta(n)$ is a very slowly growing function. In \cite{GK}, Guth and Katz developed a new partitioning theorem that has led to a revolution in combinatorial geometry. Amongst many other results, this new partitioning theorem allows one to slightly sharpen the methods from \cite{CEGSW} to show that $n$ points in $\RR^3$ span $O(n^{3/2})$ unit distances. This was done independently by Kaplan, Matou\u{s}ek, Pat\'akov\'a, and Sharir in \cite{KMSS} and by the author in \cite{Z}. 

Although many technical difficulties still abound, the ``partitioning + K\H{o}v\'ari-S\'os-Tur\'an'' technique is now well understood and has been used to make progress on a wide variety of incidence problems. With the notable exception of the Szemer\'edi-Trotter theorem for points and lines, however, this technique rarely yields bounds that are conjectured to be sharp.

In \cite{ArS}, Aronov and Sharir developed a new method for proving incidence bounds for points and circles in the plane that gives stronger results than the ``partitioning + K\H{o}v\'ari-S\'os-Tur\'an'' method. Aronov and Sharir ``cut'' a set of circles into ``pseudo-segments'' (a set of Jordan arcs are called pseudo-segments if they have the same combinatorial properties as line segments. In particular, each pair of points in the plane is incident to at most one arc). They then applied a variant of the Szemer\'edi-Trotter theorem to this set of pseudo-segments. In \cite{SZ}, Sharir and the author extended this cutting method from circles to general algebraic curves. This yielded an incidence theorem for points and curves in the plane that is stronger than the one given by the ``partitioning + K\H{o}v\'ari-S\'os-Tur\'an'' method.

The unit distance problem in $\RR^3$ can be re-cast as an incidence problem involving points and circles in $\RR^3$. In \cite{SSZ2}, Sheffer, Sharir, and the author used the ``partitioning + K\H{o}v\'ari-S\'os-Tur\'an'' method to obtain a new bound for incidences between points and circles in three dimensions (this bound is stronger than the one from \cite{ArS}, because the three-dimensionality of the point-circle arrangement is exploited). Perhaps unsurprisingly, this point-circle bound is exactly what is needed\footnote{Recovering the  $O(n^{3/2})$  bound for the unit distance problem using the  bound from \cite{SSZ2} still requires some careful arguments, since the bound from from \cite{SSZ2} contains several terms that depend on the number of circles contained in a common plane or sphere, so this degeneracy must be carefully controlled.} to recover the existing $O(n^{3/2})$  bound for the unit distance problem in $\RR^3$.

In the present paper, we extend the cuttings method developed in \cite{SZ} from plane curves to circles in $\RR^3$. This gives us a new incidence bound for points and circles in $\RR^3$ that is stronger than the one from \cite{SSZ2}, and this in turn yields a new, improved bound on the number of unit distances in $\RR^3$.

As is often the case with incidence bounds in higher dimensions, there are delicate issues regarding degeneracy. Many of the incidence bounds for points and curves in $\RR^3$ are stronger than the corresponding incidence bounds for points and curves in the plane. On the face of it, this appears suspicious, since any arrangement of points and curves in the plane can be embedded in $\RR^3$, and the number of incidences remains unchanged. To obtain stronger incidence theorems, we must prohibit these types of degenerate configurations from occurring. Much of the technical complexity of this paper comes from navigating between the possible ``degenerate'' and ``non-degenerate'' configurations of points and circles.

Finally, we remark that it is not always possible to improve upon the $O(n^{3/2})$ bound for the unit distance problem if the Euclidean metric is replaced by a different metric. In Example \ref{N32Distances} below, we give a semi-algebraic metric in which $n$ points can span $\Theta(n^{3/2})$ unit distances. The metric from Example \ref{N32Distances} can also be modified so that it is smooth (though after doing so, it is no longer semi-algebraic). 
\begin{example}\label{N32Distances}
Let $d_{*}$ be the metric whose unit ball is given by $x_3^2 \leq (1- x_1^2 - x_2^2)^2$, i.e. 
$$
d_{*}(x,y)= (x_1-y_1)^2 + (x_2-y_2)^2+|x_3-y_3|.
$$
Let $b$ be a positive integer. Let $A = \{0,1/b,2/b,\ldots,3\}$ and let $B = \{0, 1/b^2, 2/b^2,\ldots,9\}$. Let $\pts=A\times A\times B.$ Then for each point $x\in \pts$ with $ 1\leq x_1,x_2,\leq 2$ and $1\leq x_3\leq 2$, there are at least $b^2$ points $y\in\pts$ with $d_*(x,y)=1$; indeed, every point $y$ with $|y_1-x_1|\leq 1$, $|y_2-x_2|\leq 1$, and $y_3 = x_3 + (y_1-x_1)^2 + (y_2-x_2)^2-1$ will be contained in $\pts$. Thus $\pts$ spans at least $b^6 = \Omega(|\pts|^{3/2})$ unit distances.
\end{example}
The key property of the Euclidean metric that we use (which is lacking in Example \ref{N32Distances}) is that there is a one-to-one correspondence between circles in $\RR^3$ of radius $0<r<1$ and pairs of distinct unit spheres whose centers have distance less than one. Translates of the unit paraboloid $x_3=x_1^2+x_2^2$ from Example \ref{N32Distances} do not have this property: it is possible for many such translates to intersect in a common curve. 

\subparagraph{Update 3/15/2021} After this paper was published, Micha Sharir pointed out an error in the statement and proof of Lemma \ref{spaceCirclesLeadToDepthCyclesLem}, and proposed a fix. Previously, Lemma \ref{spaceCirclesLeadToDepthCyclesLem} claimed that if $C$ and $C^{\prime}$ are two circles in $\RR^3$ that intersect at two points $x$ and $y$, then under some mild assumptions, the projections of the shorter circular arcs between $x$ and $y$ to the $x_1x_2$ plane will always intersect an even number of times. In particular, Lemma \ref{spaceCirclesLeadToDepthCyclesLem} claimed that the arcs must lift to a depth cycle in $\RR^4$. While this claim was wrong as stated, it is true after applying a suitable orthogonal transformation, provided the arcs are sufficiently short. The statement and proof of Lemma \ref{spaceCirclesLeadToDepthCyclesLem}, and the proof of Corollary \ref{cuttingCirclesR3} have been updated to fix this problem. Further discussion of the issue can be found in \cite{Erratum} The rest of the paper remains unchanged.    

\section{Preliminaries}
\subsection{Notation}
Let $\pts$ be a set of points in $\RR^d$ and let $\mathcal{Z}$ be a set of sets in $\RR^d$ (usually the sets in $\mathcal{Z}$ will be algebraic varieties). We define 
$$
\mathcal{I}(\pts,\mathcal{Z})=\{(p,Z)\in\pts\times\mathcal{Z}\colon p\in Z\}.
$$ 
We define $I(\pts,\mathcal{Z})=|\mathcal{I}(\pts,\mathcal{Z})|$.

If $W\subset \RR^d$, we define
$$
W_{\mathcal{Z}}=\{Z\in\mathcal{Z}\colon Z\subset W\}.
$$

We define the set of two-rich points
$$
\pts_2(\mathcal{Z})=\{p\in\RR^d\colon p\ \textrm{is contained in at least two sets from}\ \mathcal{Z}\}.
$$

If $\omega\subset\RR^d$, we define
\begin{equation}\label{defnZOmega}
\mathcal{Z}_\omega =\{Z\in\mathcal{Z}\colon \omega\cap Z\neq\emptyset\}.
\end{equation}

Let $F$ and $G$ be functions. We say $F = O(G)$ or $F\lesssim G$ if there exists an absolute constant $A$ so that $F \leq AG$. We say $F = O_D(G)$ or $F\lesssim_D G$ if the constant $A$ can depend on the parameter $D$. We say $F\lessapprox_{\eps} G$ if there is an absolute constant $A$ so that $F\leq An^{A\eps}G$. The meaning of the variable $n$ will always be apparent from context. 

\subsection{Real algebraic varieties}\label{realAlgVarietySec}
In this section we will recall some standard results about real algebraic varieties that will be needed in later sections. A \emph{real algebraic variety} is a subset of $\RR^d$ that is the common zero-locus of a finite set of polynomials. If $V\subset\RR^d$ is an algebraic variety, the \emph{dimension} of $V$ is the largest integer $e$ so that $V$ contains a subset homeomorphic to the $e$-dimensional unit cube $(0,1)^e$. Further details can be found in \cite{BCR}. 

If $V\subset\RR^d$ is an algebraic variety of dimension $e<d$, then there exists an orthogonal projection $\pi\colon\RR^d\to\RR^{e+1}$ so that $\pi(V)$ is an algebraic variety of dimension $e$. Indeed, ``most'' orthogonal projections have this property---if we give the set of orthogonal projections $\pi\colon\RR^d\to\RR^{e+1}$ the structure of a real algebraic variety, then the set of projections for which this statement fails is contained in a proper Zariski closed subset. If $V\subset\RR^d$ can be defined using polynomials of degree $\leq t$, then the projections $\pi(V)$ discussed above can be defined using polynomials whose degree is bounded by a function that depends only on $d$ and $t$.

The set $I(V)\subset\RR[x_1,\ldots,x_d]$ is the ideal of polynomials that vanish on $V$. If $V$ is an algebraic variety of dimension $e$, then the \emph{singular locus} $V_{\operatorname{sing}}$ is the set of points of $V$ that are singular in dimension $e$. If $I(V)$ is generated by the polynomials $P_1,\ldots,P_s$, then $V_{\operatorname{sing}}$ is the set of points $p\in V$ for which the matrix $[\nabla P_1(p),\ldots,\nabla P_s(p)]$ has rank less than $d-e$; this set is independent of the choice of generators $P_1,\ldots,P_s$. 

The set of \emph{regular points} of $V$ is given by $V_{\operatorname{reg}}=V\backslash V_{\operatorname{sing}}$. If $p\in V_{\operatorname{reg}}$, then there is a (Euclidean) open set $O$ containing $p$ so that $O\cap V$ is a $e$-dimensional smooth manifold. Again, further details can be found in \cite{BCR}. We have that $\dim(V_{\operatorname{sing}})<\dim(V)$, and if $V\subset\RR^d$ can be defined by polynomials of degree at most $t$, then $V_{\operatorname{sing}}$ can be defined by polynomials of degree bounded by a function depending only on $d$ and $t$. If $P\in\RR[x_1,\ldots,x_d]$ is square-free, then each point $p\in Z(P)_{\operatorname{reg}}$ is contained in the zero-locus of exactly one irreducible component of $P$. Furthermore, there is a non-zero vector $v\in\RR^d$ so that $Z(P)_{\operatorname{sing}}$ is contained in $Z(P)\cap Z(v\cdot\nabla P)$, and the later variety has dimension at most $e-1$.

A \emph{semi-algebraic set} is a subset of $\RR^d$ that satisfies a finite list of polynomial equalities and inequalities. In particular, a real algebraic variety is a semi-algebraic set. We define the complexity\footnote{This definition is not standard, but since we will only consider semi-algebraic sets of bounded complexity, any reasonable definition of complexity will suffice.} of a semi-algebraic set $V$ to be the minimum value of $t$ so that $V$ can be defined with $\leq t$ polynomial equalities and inequalities, each of which have degree $\leq t$. If $V\subset\RR^d$ is a semi-algebraic set, then the image of $V$ under a projection $\pi:\RR^d\to\RR^e$ is also semi-algebraic. If $V$ has complexity $t$, then the image of the projection $\pi(V)$ has complexity that is bounded by a function of $t$ and $d$. 

The dimension of a semi-algebraic set can be defined analogously to that of a real algebraic set; see \cite{BCR} for details. If $V\subset\RR^d$ is a semi-algebraic set of dimension $e$ and complexity $t$ then we can write $V\subset W\cup X$, where $W$ is a smooth $e$-dimensional manifold and $X$ is a real algebraic variety of dimension strictly smaller than $e$. Furthermore, $X$ is defined by polynomials whose degree is bounded by a function of $d$ and $t$.


\subsection{Polynomial partitioning}\label{polyPartitioningSection}
The revolutionary discrete polynomial partitioning theorem developed by Guth and Katz in \cite{GK} has led to many new incidence theorems. Since then, the partitioning theorem has been extended by Guth in \cite{G1} from points to general algebraic varieties. We will use a variant of this theorem from \cite{GZ2}, which is a corollary of \cite[Theorem 0.3]{G1}. The following theorem allows us to partition sets of algebraic varieties using a partitioning polynomial $P$ in the variables $x_1,\ldots,x_d$ that is independent of some of the variables.

\begin{thm}[\!\!\cite{GZ2}, Corollary 2.4]\label{projectedPartitioning}
Let $\mathcal{V}$ be a set of $n$ algebraic varieties in $\RR^d$, each of which has dimension at most $e$ and is defined by polynomials of degree at most $t$. Let $e\leq f\leq d$. Then for each $D\geq 1$, there is a polynomial $P\in\RR[x_1,\ldots,x_d]$ of the form $P(x_1,\ldots,x_d)=Q(x_1,\ldots,x_f)$ of degree at most $D$ so that $\RR^{d}\backslash Z(P)$ is a union of $O_d(D^f)$ cells, and $O_{d,t}(nD^{e-f})$ varieties from $\mathcal{V}$ intersect each cell.
\end{thm}
When $e=0$ and $f=d$, then Theorem \ref{projectedPartitioning} is simply the original discrete polynomial partitioning theorem from \cite{GK}. When $e\geq 0$ and $f=d$, then Theorem \ref{projectedPartitioning} is the partitioning for varieties theorem from \cite{G1}.

We will also need the ``partitioning on an algebraic hypersurface'' theorem from \cite{Z}. The variant we will use here is \cite[Theorem 2.3]{Z2}. This result makes reference to a ``real ideal.'' We will not define this term here, but we will recall the following result:
\begin{lem}[Lemma 2.1, \cite{Z2}]\label{irreducibleIdealLem}
Let $P$ be an irreducible polynomial. Then there is a polynomial $Q$ with $0<\deg Q \leq \deg P$ so that $Q$ generates a real ideal and $Z(P)\subset Z(Q)$. 
\end{lem}
For our purposes, Lemma \ref{irreducibleIdealLem} says that we can assume without loss of generality that every irreducible polynomial generates an irreducible ideal.

\begin{thm}\label{partitioningOnAVariety}
Let $P\in\RR[x_1,\ldots,x_d]$ be an irreducible polynomial of degree $D$ that generates a real ideal and let $c>0$. Let $\pts\subset Z(P)$ be a set of $m$ points. Then for each $E\geq cD$, there exists a polynomial $Q\in\RR[x_1,\ldots,x_d]$ of degree at most $E$ so that $Q$ does not vanish identically on $Z(P)$, and $Z(P)\backslash Z(Q)$ is a union of $O_{c,d}(DE^{d-1})$ cells, each of which contains $O_{c,d}(m/(DE^{d-1}))$ points from $\pts$.
\end{thm}

\subsection{Connected components of sign conditions}
We will apply the partitioning theorems from Section \ref{polyPartitioningSection} to partition $\RR^d$ into cells. The following theorem controls how many of these cells an algebraic variety can intersect.
\begin{thm}[Barone-Basu \cite{BB}]\label{numberOfConnectedComponents}
Let $V\subset\RR^d$ be an algebraic variety of dimension $e$ that is defined by polynomials of degree at most $t$. Let $P\in\RR[x_1,\ldots,x_d]$ be a polynomial of degree at most $D$. Then $V\backslash Z(P)$ contains $O_{d,t}(D^e)$ connected components. 
\end{thm}

\subsection{Doubly-ruled surfaces}
A (complex) \emph{algebraic curve} is a complex algebraic variety $C\subset\CC^d$ that has dimension one. A complex surface $Z\subset\CC^3$ is said to be \emph{doubly-ruled} by curves of degree $\leq t$ if at a generic point $p\in Z$, there are at least two curves of degree $\leq t$ containing $p$ and contained in $Z$. The next lemma says that algebraic surfaces that contain many intersecting curves must be doubly-ruled by curves. 

\begin{lem}\label{manyCurvesInSurface}
For each $t\geq 1$, there is a constant $A$ so that the following holds. Let $Z\subset\CC^3$ be an irreducible variety of degree $\leq D$, and let $\mathcal{C}$ be a set of algebraic curves in $\CC^3$, each of which has degree $\leq t$ and is contained in $Z$. Suppose that no two curves share a common components, and that $|C\cap\pts_2(\mathcal{C})|\geq AD$ for at least $AD^2$ of the curves $C\in\mathcal{C}$. Then $Z$ is doubly-ruled by curves of degree $\leq t$. In particular, $\deg Z\leq 100t^2$.
\end{lem}
\begin{proof}
We will begin by recalling several results from \cite{GZ}. These results make reference to the property of being ``$(2,\mathcal{C}_t,r)$--flecnodal.'' This is a technical concept that we will not define here, since the definition can be taken as a black-box when using results from \cite{GZ}. However, a definition can be found in \cite[Section 9]{GZ} (here $\mathcal{C}_t$ is the Chow variety of curves of degree $\leq t$).
\begin{itemize}
\item Lemma 8.3 and Proposition 10.2 from \cite{GZ} says that for each $t\geq 1$ and $r\geq 1$, there is a number $A$ so that the following holds. Let $Z\subset\CC^3$ be an irreducible surface of degree $\leq D$. Suppose that there exists a set $\mathcal{C}$ of irreducible algebraic curves of degree $\leq t$ that are contained in $Z$, and that $|C\cap\pts_2(\mathcal{C})|\geq AD$ for at least $AD^2$ of the curves $C\in\mathcal{C}$. Then $Z$ is $(2,\mathcal{C}_t,r)$--flecnodal at a generic point. 

\item Theorem 8.1 from \cite{GZ} says that for each $t\geq 1$, there exists a number $r$ with the following property: If $Z\subset\CC^3$ is an irreducible surface that is $(2,\mathcal{C}_t,r)$--flecnodal at a generic point, then $Z$ is doubly-ruled by curves of degree $\leq t$. 

\item Theorem 3.5 from \cite{GZ} says that if $Z\subset\CC^3$ is an irreducible surface that is doubly-ruled by curves of degree $\leq t$, then $Z$ has degree at most $100t^2$.
\end{itemize}
Lemma \ref{manyCurvesInSurface} now follows by combining the above results.
\end{proof}

Lemma \ref{manyCurvesInSurface} can be used to understand the structure of surfaces $Z(P)\subset\RR^3$ that contain many real curves.

\begin{defn}
A set $C\subset\RR^d$ is called a \emph{real algebraic curve} if $C$ is a real algebraic variety of dimension one. To each real algebraic curve $C\in\mathcal{C}$ we can associate a complex algebraic curve $C^*$ so that $C$ is Zariski dense in $C^*$. In particular, if $C_1$ and $C_2$ are two real algebraic curves that do not share a common component, then the complex algebraic curves $C_1^*$ and $C_2^*$ also will not share a common component. If $C$ is a real algebraic curve, we define the degree of $C$ to be the degree of the complex curve $C^*$.
\end{defn}

If $C\subset\RR^d$ is a real algebraic curve, then the image of $C$ under ``most'' projections $\pi\colon\RR^d\to\RR^2$ will be an algebraic curve of the same degree. To be more precise, if we give the set of rotations $O\colon\RR^d\to\RR^d$ the structure of an algebraic variety, and if $\pi\colon\RR^d\to\RR^2$ is the projection to the first two coordinates, then the set of projections for which the image $\pi\circ O(C)$ fails to be an algebraic curve of the same degree is contained in a proper closed sub-variety. 

Note that if $\mathcal{C}$ is a set of real algebraic curves, then $|\pts_2(\mathcal{C})|\leq|\pts_2(\mathcal{C}^*)|$, where $\mathcal{C}^*=\{C^*\colon C\in\mathcal{C}\}.$

\begin{cor}\label{lowDegreeOrFewIncidences}
For each $t\geq 1$, there is a constant $A$ so that the following holds. Let $P\in\RR[x_1,x_2,x_3]$ be an irreducible polynomial of of degree $\leq D$. Let $\mathcal{C}$ be a set of $n$ real algebraic curves, each of degree at most $t$, each of which is contained in $Z(P)$, and no two of which share a common component. Then either $\deg P \leq 100t^2$, or for all but at most $AD^2$ curves $C\in\mathcal{C}$, we have $|C\cap \pts_2(\mathcal{C})|\leq AD$.
\end{cor}


We will be particularly interested in circles in $\RR^3$. The following result from \cite{SSZ2} says that for our purposes, the only interesting surfaces containing many circles are planes and spheres.
\begin{lem}[\!\cite{SSZ2}, Lemma 3.2]\label{circlesPlanesSpheres}
Let $P\in\RR[x_1,x_2,x_3]$ be an irreducible polynomial of degree  $\leq D$. Let $\mathcal{C}$ be a set of $n$ circles contained in $Z(P)$ and let $\pts\subset Z(P)$ be a set of $m$ points. Then either $Z(P)$ is a plane or sphere, or 
$
I(\pts,\circles) = O_D(m+n).
$
\end{lem}

\subsection{Existing incidence bounds for points, circles, and spheres}
When bounding the number number of incidences between points and unit spheres in $\RR^3$, we will have to deal with incidences between points and circles, which arise as the intersection of pairs of unit spheres. When understanding point-circle incidences in $\RR^3$, we will be forced to understand degenerate configurations in which many circles lie on a common sphere. Thus it will be necessary for us to deal with incidences between points and (arbitrary) spheres. 

The following two results are the current best bounds in this direction. Neither of these results are sharp, and improvements to either result would yield an improved bound on the unit distance problem. However, even if the (conjectured) best-possible point-circle and point-sphere bounds were known, this would not be enough to obtain the conjectured sharp bound on the unit distance problem using the methods from this paper. 

\begin{thm}[Aronov, Koltun, and Sharir \cite{AKS}]\label{pointCircleBound}
Let $\pts$ be a set of $m$ points and let $\circles$ be a set of $n$ circles in $\RR^3$. Then the number of point-circle incidences is
$$
O_\eps\big(m^{6/11}n^{9/11+\eps}+m^{2/3}n^{2/3}+m+n\big).
$$
\end{thm}

\begin{defn}\label{defnNonDegenerate}
Let $S$ be a sphere and let $0<\alpha< 1$. We say that $S$ is $\alpha$-non-degenerate with respect to a set of points $\pts$ if for every circle $C\subset S$ we have 
$$
|C\cap S\cap\pts|\leq\alpha |S\cap\pts|.
$$
\end{defn}
\begin{thm}[Apfelbaum and Sharir \cite{AS}]\label{richSphereThm}
Let $\pts$ be a set of $m$ points in $\RR^3$. Then for each $0<\alpha<1$, the number of $k$--rich, $\alpha$-non-degenerate spheres is 
$$
O_\alpha(m^4k^{-5}+mk^{-1}).
$$
\end{thm}
\section{Cutting circles in $\RR^3$ into pseudosegments}
In \cite{SZ}, Sharir and the author showed that $n$ algebraic plane curves can be cut into $O(n^{3/2})$ pseudo-segments. In this section we will extend this result from plane curves to circles in $\RR^3$. It is possible that a similar result holds for general algebraic curves in $\RR^3$, but the proof from \cite{SZ} uses topological arguments that do not generalize from plane curves to (general) space curves.

\subsection{Cuttings}
\begin{defn}
A set $\gamma\subset\RR^d$ is called an open (resp. closed) Jordan arc if it is homeomorphic to the open interval $(0,1)$ (resp. closed interval $[0,1])$. Unless stated otherwise, all Jordan arcs that we refer to will be open.
\end{defn}
\begin{defn}
Let $\Gamma$ be a set of Jordan arcs in $\RR^d$. We say $\Gamma$ is a set of pseudo-segments if for every pair of points $x,y\in\RR^d$, there is at most one curve from $\Gamma$ that contains both $x$ and $y$.
\end{defn}

\begin{defn}\label{defnOfCutting}
Let $\mathcal{C}$ be a set of algebraic space curves in $\RR^d$, no two of which share a common component. Let $\Gamma$ be a set of Jordan arcs, each pair of which have finite intersection. We say that $\Gamma$ is a cutting of $\mathcal{C}$ if each curve in $\mathcal{C}$ can be written as a union of finitely many arcs from $\Gamma$, plus finitely many points. If $|\Gamma|=N$, we say that $\Gamma$ is a cutting of $\mathcal{C}$ into $N$ pieces.
\end{defn}
Rather than keeping track of the Jordan arcs obtained by cutting algebraic curves, it will sometimes be easier to keep track of the points that are removed from each curve to obtain the cutting. Thus if $\mathcal{C}$ is a set of algebraic curves in $\RR^d$, it will sometimes be helpful to think of a cutting as a set $\mathcal{D}\subset \mathcal{C}\times\RR^d$ with the following properties:
\begin{itemize}
\item For each $(C,p)\in\mathcal{D}$, $p\in C$.
\item For each $C\in\mathcal{C}$, let $\pts_C=(\{C\}\times\RR^d)\cap\mathcal{D}$. Then each connected component of $C\backslash \pts_C$ is a Jordan arc.
\end{itemize}
 By \cite[Lemma 2.2]{SZ}, the number of connected components of $C\backslash\pts_C$ is $O_t(|\pts_C|+1)$ (\cite[Lemma 2.2]{SZ} refers to plane curves, but the same proof works for space curves). Thus if we define 
$$
 \Gamma=\bigcup_{C\in\mathcal{C}}\{\gamma\colon \gamma\ \textrm{is a connected component of}\ C\backslash\pts_C\},
$$
then $\Gamma$ is a cutting of $\mathcal{C}$ in the sense of Definition \ref{defnOfCutting} into $O_t(|\mathcal{D}|+|\mathcal{C}|)$ pieces.

Conversely, if $\Gamma$ is a cutting of $\mathcal{C}$ into $N$ Jordan arcs, then we can represent the cutting $\Gamma$ as a set $\mathcal{D}\subset\mathcal{C}\times\RR^d$ with $|\mathcal{D}|\leq 2|\Gamma|$; simply define 
$$
\mathcal{D}=\bigcup_{\gamma\in\Gamma}\{(p,C)\colon \gamma\subset C,\ p\ \textrm{is an endpoint of}\ \gamma\}.
$$
\subsection{Vertical hypersurfaces}\label{vertHypersurfaceSec}
\begin{defn}
A \emph{vertical hypersurface} is an algebraic variety of the form $Z = Z(P)\subset\RR^d$, where $P\in\RR[x_1,\ldots,x_d]$ is independent of the $x_d$ variable. If $Z$ is a vertical hypersurface then $(x_1,\ldots,x_d)\in Z$ if and only if $\{(x_1,x_2,\ldots,x_{d-1},y)\colon y\in\RR\}\subset Z$.
\end{defn}

The following is an immediate consequence of Corollary \ref{lowDegreeOrFewIncidences}.
\begin{lem}\label{verticalHyperDoublyRuled}
Let $\pi\colon \RR^4\to\RR^3$ be the projection to the first three coordinates. Let $Z=Z(P)$ be an irreducible vertical hypersurface in $\RR^4$ of degree $\leq D$. Let $\mathcal{C}$ be a set of space curves in $\RR^4$ of degree $\leq t$, and suppose that no two curves have projections to the $(x_1,x_2,x_3)$ hyperplane that share a common component. Then either $\pi(Z)$ is doubly-ruled by curves of degree $\leq t$, or the curves in $\mathcal{C}$ can be cut into $O_{D,t}(n)$ Jordan arcs whose projections to the $(x_1,x_2,x_3)$ hyperplane are disjoint.
\end{lem}

Let $P\in\RR[x_1,\ldots,x_4]$. For each point $x=(x_1,x_2,x_3,x_4)\in\RR^4$, define $h_P(x)$ to be the number of points of intersection of $Z(P)$ with the closed ray $\{(x_1,x_2,x_3,y)\in\RR^4\colon y\leq x_4\}$. Note that for each $x\in\RR^4$, $h_P(x)$ is either a non-negative integer or is $\infty$. 

Let $W_P=Z(P)\cap Z(\partial_{x_4} P)$, let $\pi\colon\RR^4\to\RR^3$ be the projection to the first three coordinates, and let 
\begin{equation}\label{defnVP}
V_P = \pi^{-1}\big(\ \overline{\pi(W_p)}\ \big), 
\end{equation}
where $\overline{\phantom{1}\cdot\phantom{1}}$ denotes Zariski closure. Note that $V_P$ is a vertical hypersurface defined by a polynomial of degree at most $D^2$. By \cite[Section 6]{SZ}, the function $h_P(x)$ is constant on each (Euclidean) connected component of $\RR^4\backslash (Z(P)\cup V_P)$ (the result in \cite{SZ} is the analogous statement in $\RR^3$ rather than $\RR^4$, but the proof is identical; the same proof works in any dimension $d\geq 3$).

\subsection{Lifting space curves to $\RR^4$}
In this section we will describe a transformation that sends space curves in $\RR^3$ to space curves in $\RR^4$. The extra dimension will encode information about the slope of the curve.

Let $C$ be an irreducible curve in $\RR^3$ and let $\underline{C}$ be the projection of $C$ to the $(x_1,x_2)$ plane. Applying a generic rotation if necessary, we can assume that the image of this projection is an irreducible algebraic curve, and that the fiber above a generic point in the image of the projection has cardinality one. We can also assume that the degree of $\underline{C}$ is the same as that of $C$. 

Let $f^{\underline{C}}\in \RR[x,y]$ be a square-free polynomial with $Z(f^{\underline{C}})=\underline{C}$. Following the strategy in \cite{ESZ} and \cite{SZ} consider the set
$$
\{(x_1,\ldots,x_4)\in\RR^4\colon (x_1,x_2,x_3)\in C,\ x_4\partial_{x_1}  f^{\underline{C}} = \partial_{x_2} f^{\underline{C}} \}.
$$ 
As discussed in Section 3.3 of \cite{ESZ}, this set is the union of a set of vertical lines (i.e. lines whose projection to the $(x_1,x_2,x_3)$ hyperplane consists of a point), plus an irreducible space curve in $\RR^4$ that is not a vertical line. Call this irreducible curve $\beta(C)$. Intuitively, if $(x_1,x_2,x_3)\in C$, then $(x_1,x_2,x_3,x_4)\in \beta(C)$ if $x_4$ is the slope of $\underline{C}$ at the point $(x_1,x_2)$.

If $\mathcal{C}$ is a set of irreducible algebraic curves in $\RR^3$, define 
$$
\beta(\mathcal{C})=\{\beta(C)\colon C\in\mathcal{C}\}.
$$
Applying a generic rotation if necessary, we will assume that each set in $\beta(\mathcal{C})$ is an irreducible algebraic curve. 

\subsection{Depth cycles}
\begin{defn}
Let $\gamma$ and $\gamma^\prime$ be closed Jordan arcs in $\RR^d$. We say that $\gamma_1$ and $\gamma_2$ form a \emph{depth cycle} if there exist points $(x_1,\ldots,x_d), (y_1,\ldots,y_d)\in\gamma$ and $(x_1,\ldots,x_{d-1},x_d^\prime), (y_1,\ldots,y_{d-1},y_d^\prime)\in\gamma^\prime$ so that either $x_d\leq x_d^\prime$ and $y_d\geq y_d^\prime$, or $x_d\geq x_d^\prime$ and $y_d\leq y_d^\prime$. We say that $\gamma_1$ and $\gamma_2$ form a \emph{proper depth cycle} if one of these inequalities is strict. 
\end{defn}

\begin{defn}
We say that the closed Jordan arcs $\gamma$ and $\gamma^\prime$ form a \emph{minimal depth cycle} if $(x_1,\ldots,x_d)$ and $(y_1,\ldots,y_d)$ are the endpoints of $\gamma$; $(x_1,\ldots,x_{d-1},x_d^\prime)$ and $(y_1,\ldots,y_{d-1},y_d^\prime)$ are the endpoints of $\gamma^\prime$; and the projections of $\gamma$ and $\gamma^\prime$ to the $(x_1,\ldots,x_{d-1})$ hyperplane intersect only at the points $(x_1,\ldots,x_{d-1})$ and $(y_1,\ldots,y_{d-1})$.
\end{defn}
If $\gamma$ and $\gamma^\prime$ form a depth cycle, then there always exist closed Jordan arcs $\tilde\gamma\subset\gamma$ and $\tilde\gamma^\prime\subset\gamma^\prime$ that form a minimal depth cycle. The choice of $\tilde\gamma$ and $\tilde\gamma^\prime$ might not be unique, however.



In \cite{SZ}, Sharir and the author proved the following result:
\begin{thm}[\cite{SZ}, Theorem 1.2]\label{R3DepthCycleCutting}
Let $\mathcal{C}$ be a set of $n$ algebraic curves in $\RR^3$, each of degree at most $t$. Suppose that no two curves have projections to the $(x_1,x_2)$ plane that share a common component. Then the curves in $\mathcal{C}$ can be cut into $O_{t,\eps}(n^{3/2+\eps})$ Jordan arcs so that the arcs contain no proper depth cycles.
\end{thm}

Since vertical depth cycles are preserved under projections of the form $(x_1,x_2,x_3,x_4)\mapsto(\pi(x_1,x_2,x_3),x_4)$ where $\pi\colon\RR^3\to\RR^2$ is a projection, Theorem \ref{R3DepthCycleCutting} also allows us to cut algebraic curves in $\RR^4$ into Jordan arcs so that all depth cycles are eliminated.

\begin{cor}\label{corOfSZ}
Let $\mathcal{C}$ be a set of $n$ algebraic curves in $\RR^4$, each of degree at most $t$. Suppose that no two curves have projections to the $(x_1,x_2,x_3)$ hyperplane that share a common component. Then the curves in $\mathcal{C}$ can be cut into $O_{t,\eps}(n^{3/2+\eps})$ Jordan arcs so that the arcs contain no proper depth cycles.
\end{cor}

If the algebraic curves satisfy certain non-degeneracy conditions, however, then a stronger bound is possible.

\begin{thm}\label{cuttingSpaceCurvesLem}
For each integer $t$ and each $\eps>0$, there is a constant $A$ so that the following holds. Let $\mathcal{C}$ be a set of $n$ algebraic curves in $\RR^4$, each of degree at most $t$. Suppose that no two curves have projections to the $(x_1,x_2,x_3)$ hyperplane that share a common component. Then there are sets $\mathcal{C}=\mathcal{C}_1\sqcup\mathcal{C}_2$, a set $\mathcal{Z}$ of vertical hypersurfaces in $\RR^4$ and a cutting $\mathcal{D}\subset\mathcal{C}_2\times\RR^d$ with the following properties. 
\begin{itemize}
\item Each hypersurface $Z\in\mathcal{Z}$ has degree at most $100t^2$.
\item $|\mathcal{Z}|\leq 2n^{1/3-\eps}$.
\item Each curve in $\mathcal{C}_1$ is contained in (at least one) vertical hypersurface from $\mathcal{Z}$.
\item Each $Z\in\mathcal{Z}$ contains at least $n^{2/3+\eps}$ curves from $\mathcal{C}_1$.
\item $|\mathcal{D}|\leq An^{4/3+3\eps}$.
\item $\mathcal{D}$ is a cutting of $\mathcal{C}_2$ into Jordan arcs; each arc is disjoint from each hypersurface in $\mathcal{Z}$. 
\item The arcs from the cutting $\mathcal{D}$ contain no proper depth cycles.
\end{itemize}
\end{thm}
\begin{proof}
We will prove the result by induction on $n$. If $n$ is small then the result is immediate provided we choose the constant $A$ sufficiently large; simply choose $\mathcal{Z}=\emptyset$ and cut the curves from $\mathcal{C}$ at each singular point of each curve and at each point where two or more curves have the same projection to the $(x_1,x_2,x_3)$ hyperplane.

We will now discuss the induction step. By Theorem \ref{projectedPartitioning} (with $d=4,\ e=1,$ and $f=4$) there is a partitioning polynomial $P\in\RR[x_1,x_2,x_3,x_4]$ of degree $D$ so that $\RR^4\backslash Z(P)$ is a union of at most $A_0D^4$ cells and at most $A_0nD^{-3}$ curves from $\mathcal{C}$ intersect each cell, where $A_0$ is a constant that depends only on $t$. The number $D$ will be chosen later; it will depend on $\eps$ and $t$, but not on $n$.

Let $\gamma$ and $\gamma^\prime$  be closed Jordan arcs in $\RR^4$, each of which is contained in a curve from $\mathcal{C}$ (though $\gamma$ and $\gamma^\prime$ need not be contained in the same curve), and suppose that $\gamma$ and $\gamma^\prime$ form a minimal proper depth cycle. Then at least one of the following four things must happen:
\begin{enumerate}
\item\label{sameCell} $\gamma$ and $\gamma^\prime$ are entirely contained in the same cell.
\item\label{differentCells} $\gamma$ and $\gamma^\prime$ are each entirely contained in a cell, but these cells are different.
\item\label{properlyIntersect} At least one of $\gamma$ or $\gamma^\prime$ intersects $Z(P)$ but is not contained in $Z(P)$.
\item\label{containedIn} Both of $\gamma$ and $\gamma^\prime$ are contained in $Z(P)$.
\end{enumerate}

We will cut the curves in $\mathcal{C}$ to eliminate each of these types of depth cycles.

\subsubsection{Eliminating depth cycles of Type \ref{sameCell}}
We will first describe a procedure to eliminate all depth cycles of Type \ref{sameCell}. Let $\Omega$ be the set of cells. For each cell $\omega\in\Omega$, apply the induction hypothesis to $\mathcal{C}_{\omega}$ (recall from \eqref{defnZOmega} that this is the set of curves that intersect $\omega$). We obtain sets $\mathcal{Z}_{\omega}$, $\mathcal{C}_{\omega,1},\ \mathcal{C}_{\omega,2}$, and $\mathcal{D}_\omega$ so that:

\begin{itemize}
\item $\mathcal{C}_{\omega}=\mathcal{C}_{\omega,1}\sqcup\mathcal{C}_{\omega,2}$.
\item Each vertical hypersurface $Z\in\mathcal{Z}_\omega$ has degree at most $100t^2$.
\item $|\mathcal{Z}_\omega|\leq 2|\mathcal{C}_{\omega}|^{1/3-\eps}\leq (A_0n/D^3)^{1/3-\eps}$.
\item Each curve in $\mathcal{C}_{\omega,1}$ is contained in (at least one) vertical hypersurface from $\mathcal{Z}_\omega$.
\item Each $Z\in\mathcal{Z}_\omega$ contains at least $|\mathcal{C}_{\omega}|^{2/3+\eps}$ curves from $\mathcal{C}_{\omega,1}$.
\item $|\mathcal{D}_\omega|\leq A|\mathcal{C}_\omega|^{4/3+3\eps}\leq  A(A_0n/D^3)^{4/3+3\eps}$.
\item $\mathcal{D}_\omega$ is a cutting of $\mathcal{C}_{\omega,2}$ into Jordan arcs; each arc is disjoint from each hypersurface in $\mathcal{Z}_\omega$.
\item The arcs from the cutting $\mathcal{D}_\omega$ contain no proper depth cycles.
\end{itemize}
Define
$$
\mathcal{D}_1=\bigcup_{\omega}\mathcal{D}_{\omega}.
$$
We have that
$$
|\mathcal{D}_1|\leq (A_0D^4) A(A_0n/D^3)^{4/3+3\eps} = (A_0^{7/3+3\eps} D^{-9\eps}) (An^{4/3+3\eps}).
$$
If $D$ is chosen sufficiently large (depending only on $\eps$ and $A_0$, which in turn depends only on $t$), then 
$$
|\mathcal{D}_1|\leq \frac{A}{10}n^{4/3+3\eps}.
$$

Define $\mathcal{Z}_1 = \bigcup_{\omega}\mathcal{Z}_\omega$. It would be good if $|\mathcal{Z}_1|\leq 2n^{1/3-\eps}$, but all we know is that $|\mathcal{Z}_1| \leq A_0^{4/3-\eps}D^{3+3\eps}n^{1/3-\eps};$ this bound is too weak to close the induction. We will ``fix'' this issue below. 

Let $V_P$ be the vertical hypersurface defined in \eqref{defnVP}. For each component $V^\prime\subset V_P$ that has degree $>100t^2$, use Lemma \ref{verticalHyperDoublyRuled} to cut the curves of $\mathcal{C}_{V^\prime}$ into $O_{t,D}(n)$ Jordan arcs, no two of which form a proper depth cycle. Denote this cutting by $\mathcal{D}_{V^\prime}$ and let $\mathcal{D}_2=\bigcup \mathcal{D}_{V^\prime},$ where the union is taken over all irreducible components of $V_{P}$ that have degree $>100t^2$. We have $|\mathcal{D}_2|=O_{D,t}(n)$. 

Let $\mathcal{Z}_2$ be the union of $\mathcal{Z}_1$ and the irreducible components of $V_P$ that have degree $\leq 100d^2$. We have that 
$$
|\mathcal{Z}_2| \leq D^2 + A_0^{4/3-\eps}D^{3+3\eps}n^{1/3-\eps} \leq 2A_0^{4/3-\eps}D^{3+3\eps}n^{1/3-\eps}.
$$

Define
$$
\mathcal{Z} = \{Z\in\mathcal{Z}_2\colon |\mathcal{C}_Z|\geq n^{2/3+\eps}  \}.
$$
Observe that if $Z_1,Z_2\in\mathcal{Z}_2$ are distinct, then the projection of $Z_1\cap Z_2$ to the $(x_1,x_2,x_3)$ hyperplane is an algebraic curve of degree at most $A_t = (100t^2)^2$. In particular, since no two curves have $(x_1,x_2,x_3)$ projections that share a common component,  at most $A_t$ curves $C\in\mathcal{C}$ can satisfy $C\subset Z_1\cap Z_2$. Combining this observation with the inclusion-exclusion principle, we conclude that 
\begin{equation}\label{inclusionExclusionEqn}
\begin{split}
n^{2/3+\eps}|\mathcal{Z}|&\leq \sum_{Z\in\mathcal{Z}_2}|\mathcal{C}_Z|\\
&\leq \Big|\bigcup_{Z\in\mathcal{Z}_2}\mathcal{C}_Z\Big| + \sum_{Z_1,Z_2\in\mathcal{Z}_2}|\mathcal{C}_{Z_1}\cap \mathcal{C}_{Z_2}|\\
&\leq n + A_t|\mathcal{Z}_2|^2 \\
&\leq n + A_t|\mathcal{Z}_2|^2\\
&\leq n + A_tA_0^{8/3}D^{6}n^{2/3}\\
&\leq 2n,
\end{split}
\end{equation}
where on the final line we used the fact that if $n$ is sufficiently large (compared to $t$ and $\eps$), then $A_tA_0^{8/3}D^{6}n^{2/3}\leq n$. Thus $|\mathcal{Z}|\leq 2n^{1/3-\eps}$. Define 
$$
\mathcal{C}_1=\bigcup_{Z\in\mathcal{Z}}\mathcal{C}_Z,
$$
and define $\mathcal{C}_2=\mathcal{C}\backslash\mathcal{C}_1$. Note that $\mathcal{Z}$ and $\mathcal{C}_1$ satisfy the first four requirements from Lemma \ref{cuttingSpaceCurvesLem}. 

For each $Z\in\mathcal{Z}_2\backslash\mathcal{Z}$, use Corollary \ref{corOfSZ} to cut the curves in $\mathcal{C}_Z$ into at most
$$
A_{t,\eps}|\mathcal{C}_Z|^{3/2+\eps}\leq A_{t,\eps}(n^{2/3+\eps})^{3/2+\eps}=A_{t,\eps}n^{1+3\eps}
$$ 
Jordan arcs, so that the resulting collection of arcs contains no depth cycles. 
%
The total number of cuts required to perform this step is 
$$
\leq A_{t,\eps}\big(n^{1+3\eps}|\mathcal{Z}_2\backslash\mathcal{Z}|\big)=A_{t,\eps}\big(D^{3+3\eps}n^{1/3-\eps}n^{1+3\eps}\big)\leq\frac{1}{5}n^{4/3+3\eps}.
$$
Call the resulting cutting $\mathcal{D}_3$.

Observe that at this point, if $C,C^\prime\in\mathcal{C}_2$, and if $\gamma\subset C$ and $\gamma^\prime\subset C^\prime$ are Jordan arcs contained inside the same cell that form a proper depth cycle, then there must either be a point $x\in\gamma$ with $(x,C)\in\mathcal{D}_1\cup\mathcal{D}_2\cup\mathcal{D}_3$, or there must be a point $y\in\gamma^\prime$ with $(y,C^\prime)\in \mathcal{D}_1\cup\mathcal{D}_2\cup\mathcal{D}_3$. In other words, for each cell $\omega\in\Omega$, all depth cycles of Type \ref{sameCell} have been eliminated. It remains to eliminate the other types of depth cycles. 
\subsubsection{Eliminating depth cycles of Type \ref{differentCells}, \ref{properlyIntersect}, and \ref{containedIn}}
Let 
\begin{equation*}
\begin{split}
\mathcal{D}_4 &= \{(C,x)\in\mathcal{C}_2\times\RR^4\colon C\not\subset V_P,\ x\in C\cap V_P\},\\
\mathcal{D}_5 & = \{(C,x)\in\mathcal{C}_2 \times\RR^4\colon C\not\subset Z(P),\ x\in C\cap Z(P)\}.
\end{split}
\end{equation*}
We have $|\mathcal{D}_4\cup\mathcal{D}_5|=O_{D,t}(n)$. 

The cutting $\mathcal{D}_5$ eliminates all depth cycles of Type \ref{properlyIntersect}. Next we will argue that the cutting $\mathcal{D}_4$ eliminates all depth cycles of Types \ref{differentCells} and \ref{containedIn}. Let $\gamma$ and $\gamma^\prime$ be closed Jordan arcs that form a minimal depth cycle, and suppose that either each of $\gamma$ and $\gamma^\prime$ are entirely contained in distinct cells, or each of $\gamma$ and $\gamma^\prime$ are contained in $Z(P)$. If $x$ and $y$ are the endpoints of $\gamma$ and if $x^\prime, y^\prime$ are the endpoints of $\gamma^\prime$, then after interchanging the roles of $\gamma$ and $\gamma^\prime$ if necessary, we have $h(x)\geq h(x^\prime)$ and $h(y)\leq h(y^\prime)$, and at least one of these inequalities must be strict. In particular, there must exist a point $z_0$ on either $\gamma$ or $\gamma^\prime$ where the function $h(z)$ changes value. Suppose the point is on $\gamma$. Since every point of this type is contained in $V$, we have that $(C,z_0)\in\mathcal{D}_4$, where $C$ is the curve containing $\gamma$. If instead the point is on $\gamma^\prime$, then an identical argument applies with $C^\prime$ in place of $C$.

Finally, let 
$$
\mathcal{D}_6=\bigcup_{Z\in\mathcal{Z}}\{(C,x)\colon\ x\ \textrm{is a singular point of}\ C,\ \textrm{or}\ \pi(x)\ \textrm{is a}\ x_2\textrm{-extremal point of}\ \underline C\},
$$
where $\pi\colon\RR^3\to\RR^2$ is the projection to the $(x_1,x_2)$ plane. We have that $|\mathcal{D}_6|=O_{t}(|\mathcal{C}_2|)$. Let $\mathcal{D}=\mathcal{D}_1\cup\ldots\mathcal{D}_6$. If $A$ is chosen sufficiently large (depending only on $\eps$ and $t$), then $|\mathcal{D}|\leq An^{4/3+3\eps}$. By construction, $\mathcal{D}$ is a cutting of the curves in $\mathcal{C}_2$ into Jordan arcs so that all depth cycles are eliminated. Furthermore, each arc in the cutting is disjoint from each surface in $\mathcal{Z}$. This completes the induction step and finishes the proof. 
\end{proof}
The main motivation for \cite[Theorem 1.2]{SZ} is that there is a transformation from plane curves in $\RR^2$ to space curves in $\RR^3$ so that lenses (pairs of curves that intersect at two common points) become depth cycles. Thus \cite[Theorem 1.2]{SZ} allows one to cut a set of $n$ algebraic plane curves into $O(n^{3/2+\eps})$ pseudo-segments. 

We would like to do something similar in $\RR^3$. Recall that our procedure for lifting space curves from $\RR^3$ to $\RR^4$ privileges the $x_1x_2$ plane. This can be problematic, because additional intersections can be introduced when we project a pair of circles $C,C^{\prime}$ from $\RR^3$ to the $x_1x_2$ plane. The next lemma says that there is a finite set $\mathcal{V}$ of unit vectors, so that for each pair of space circles $C,C^\prime$ that intersect at two points, there is vector $v\in\mathcal{V}$ so that after rotating $\RR^3$ to make $v$ the $x_3$ axis, the problem described above is resolved.

\begin{lem}\label{spaceCirclesLeadToDepthCyclesLem}
There is a finite set of unit vectors $\mathcal{V}\subset\RR^3$ so that the following holds. Let $C$ and $C^\prime$ be circles in $\RR^3$ that intersect at the two distinct points $x$ and $y$. Let $\tilde C\subset C$ be the shorter (closed) arc of $C$ with endpoints $x$ and $y$, and similarly for $\tilde C^\prime\subset C^\prime$. Suppose that the length of $\tilde C$ is at most $1/100$ the circumference of $C$, and similarly for $\tilde C^\prime$. Then there is a vector $v\in\mathcal{V}$ so that the orthogonal projection $\pi_{v^\perp}\colon \RR^3\to v^\perp$ is injective on $\tilde C \cup\tilde C^{\prime}$. 
\end{lem}
\begin{proof}
Let $\mathcal{V}$ be a set of unit vectors, so that every vector in $\RR^3$ makes an angle $\leq 1$ degree with one of the vectors in $\mathcal{V}$. The set $\mathcal{V}$ is clearly finite. Next, let $C$ and $C^\prime$ be circles in $\RR^3$ that intersect at the two distinct points $x$ and $y$, and let $\tilde C\subset C$, $\tilde C^{\prime}\subset C^\prime$ be the shorter (closed) arcs of $C$ and $C^\prime$, respectively, with endpoints $x$ and $y$. Suppose that the length of $\tilde C$ is at most $1/100$ the circumference of $C$, and similarly for $\tilde C^\prime$. We will show that there exists a vector $v\in\mathcal{V}$ so that the projection $\pi_{v^\perp}\colon \RR^3\to v^\perp$ is injective on $\tilde C \cup\tilde C^{\prime}$. 

First, if $C$ and $C^\prime$ are contained in a common plane $S\subset\RR^3$, select $v\in\mathcal{V}$ to be any vector not parallel to $S$. Since $v$ is not parallel to $S$, the projection $\pi_{v^\perp}\colon S\to v^\perp$ is injective, and thus $\pi_{v^\perp}$ is injective on $\tilde C \cup\tilde C^{\prime}$.

Next, if $C$ and $C^\prime$ are not coplanar, let $S\subset\RR^3$ be the unique sphere containing $C$ and $C^\prime$; such a sphere must exist since $C$ and $C^\prime$ intersect at two points. Select $v\in\mathcal{V}$ to be a vector that makes angle $\leq 1$ degree with the normal vector of $S$ at $x$. Observe that 
\begin{equation*}
\operatorname{radius}(S)\geq\max(\operatorname{radius}(C),\ \operatorname{radius}(C^\prime)).
\end{equation*}
Let $\tilde S\subset S$ be the spherical cap of diameter $(1/10)\operatorname{radius}(S)$ centered at $x$. We have that $\tilde C \subset \tilde S$ and $\tilde C^\prime\subset \tilde S$.
The projection $\pi_{v^\perp}\colon \tilde S\to v^{\perp}$ is injective (this is a simple geometry exercise), and thus $\pi_{v^\perp}$ is injective on $\tilde C \cup\tilde C^{\prime}$.
\end{proof}

\begin{cor}\label{cuttingCirclesR3}
Let $\mathcal{C}$ be a set of $n$ circles in $\RR^3$. Suppose that at most $B$ circles can lie in a common algebraic variety of degree $\leq 400$. Then the circles in $\mathcal{C}$ can be cut into
$
O_{\eps}(n^{4/3+\eps} + nB^{1/2+\eps}  )
$
pseudo-segments.
\end{cor}

\begin{rem}
In fact, we can replace the requirement that at most $B$ circles lie in a variety of degree $\leq 400$ with the requirement that at most $B$ circles lie in a surface that is doubly ruled by circles. The set of all such surfaces has been classified in \cite{Sk}. However this classification will not be relevant for our proof. The only surfaces that will concern us are planes and spheres; these are the only low degree surfaces that contain many circles through each point.
\end{rem}
\begin{proof}
First, cut each circle in $\mathcal{C}$ into 100 arcs of equal length; we will call this the first cutting.

Next, let $\mathcal{V}$ be the set of unit vectors from Lemma \ref{spaceCirclesLeadToDepthCyclesLem}. For each $v\in\mathcal{V}$, let $O_v\colon\RR^3\to\RR^3$ be an orthogonal transformation sending $v$ to the vector $(0,0,1)$. Define $O_v(\mathcal{C}) = \{O_v(C)\colon C\in\mathcal{C}\}$; this is a set of circles in $\RR^3$, and $\{ \pi_{(0,0,1)^\perp}(C)\colon C\in O_v(\mathcal{C}) \}$ is a collection of ellipses in the plane (by perturbing the vectors in $\mathcal{V}$ slightly if necessary, we can ensure that none of the projected circles are line segments). For each $v\in\mathcal{V}$, we will cut each of these ellipses at their two points of vertical tangency. This corresponds to $2|\mathcal{V}| n=O(n)$ cuts to the original set of circles. We will call this the second cutting.

Now, observe that if $C,C^{\prime}\in\mathcal{C}$ are circles that intersect at two distinct points $x,y\in\RR^3$, then either (A): at most one of the four arcs from $C$ and $C^{\prime}$ connecting $x$ and $y$ remains uncut, or (B): the (shorter) arc $\tilde C \subset C$ between $x$ and $y$ has length at most $1/100$ the circumference of $C$, and similarly for the shorter arc $\tilde C^\prime \subset C^\prime$. If (A) occurs, then the first cutting has ensured that at most one of the segments from $C$ or $C^\prime$ contain both $x$ and $y$. 

Next, suppose that (B) occurs, and let $\tilde C\subset C$ and $\tilde C^{\prime}\subset C^{\prime}$ be the shorter arcs of $C$ and $C^{\prime}$, respectively, with endpoints $x$ and $y$. By Lemma \ref{spaceCirclesLeadToDepthCyclesLem}, there is a unit vector $v\in\mathcal{V}$ so that the projection $\pi_{v^\perp}\colon \RR^3\to v^\perp$ is injective on $\tilde C \cup\tilde C^{\prime}$. We will call $v$ the vector associated to the pair $(C,C^\prime)$. If there is more than one such vector $v\in\mathcal{V}$, we will choose one arbitrarily. Define $\pi\colon\RR^3\to\RR^2$ to be the projection $(x_1,x_2,x_3)\mapsto(x_1,x_2)$. Either (B.1) at least one of the ellipse segments $\pi(O_v(\tilde C))$ or $\pi(O_v(\tilde C^\prime))$ contains a point of vertical tangency, or (B.2) neither of these ellipse segments contain a point of vertical tangency. If (B.1) occurs, then the second cutting has ensured that at most one of the segments from $C$ or $C^\prime$ contain both $x$ and $y$. 

If (B.2) occurs, then the closed Jordan arcs $\gamma = \pi(O_v(\tilde C))\subset\RR^2$ and $\gamma^\prime = \pi(O_v(\tilde C^\prime))\subset\RR^2$ intersect only at the two points $\pi(O_v(x))$ and $\pi(O_v(y)).$ This means that the two arcs form a (planar) lens. Furthermore, since neither of these arcs have a point of vertical tangency, by  \cite[Lemma 2.6]{SZ}, exactly one of the following two things must happen: $\gamma$ has larger slope then $\gamma^\prime$ at $\pi(O_v(x))$ and smaller slope at $\pi(O_v(y))$, or $\gamma$ has smaller slope then $\gamma^\prime$ at $\pi(O_v(x))$ and larger slope at $\pi(O_v(y))$. This means that the lifted curves $\beta(\tilde C)$ and $\beta(\tilde C^{\prime})$ form a depth cycle. The next step is to perform additional cuts to disrupt all depth cycles of this type.

For each $v\in\mathcal{V}$, apply Theorem \ref{cuttingSpaceCurvesLem} to the set of circles $O_v(\mathcal{C})$ (recall that Theorem \ref{cuttingSpaceCurvesLem} privileges the $x_1x_2$ plane, which has been rotated by the orthogonal transformation $O_v$). We obtain a cutting of the circles in $O_v(\mathcal{C})$ and a collection of surfaces, each of which contain many circles from $O_v(\mathcal{C})$. Let $\mathcal{D}_v$ and $\mathcal{Z}_v$ be the image of this cutting and the surfaces in $\mathcal{Z}$, respectively, under the inverse transformation $O_v^{-1}$. Thus $\mathcal{D}_v$ is a cutting of the circles in $\mathcal{C}$ (as opposed to the circles in $O_v(\mathcal{C})$), and $\mathcal{Z}_v$ is a collection of surfaces in $\RR^3$, each of which contain many circles from $\mathcal{C}$. 

Let 
\[
\mathcal{D} = \bigcup_{v\in\mathcal{V}}\mathcal{D}_v,\quad \mathcal{Z} = \bigcup_{v\in\mathcal{V}}\mathcal{Z}_v.
\]
By Theorem \ref{cuttingSpaceCurvesLem} we have $|\mathcal{D}|\leq |\mathcal{V}| O_{\eps}(n^{4/3+\eps})= O_{\eps}(n^{4/3+\eps})$. We will call these cuts the third cutting.

If $B\leq n^{2/3+\eps}$, then $\mathcal{Z}$ is empty. Otherwise, for each $Z\in\mathcal{Z}$, cut the circles contained in $Z$ into pseudo-segments; this can be done with $|\mathcal{C}_Z|^{3/2+\eps}$ cuts. We will call this the fourth cutting. Note that $\sum_{Z\in\mathcal{Z}}|\mathcal{C}_Z|=O(n)$, and each set $\mathcal{C}_Z$ has size at most $B$. Thus the total number of cuts required to perform the fourth cutting is
\[
\sum_{Z\in\mathcal{Z}}O_\eps(|\mathcal{C}_Z|^{3/2+\eps})=O_\eps (nB^{1/2+\eps}).
\]

Next, for each circle $C\in\mathcal{C}$, cut $C$ at each point where it properly intersects a surface $Z \in\mathcal{Z}$. We will call this the fifth cutting. If $\mathcal{Z}$ is non-empty, then $B> n^{2/3+\eps}$ and thus $|\mathcal{Z}| = O(n/B)$, so $|\mathcal{Z}| = O(n/B) = O(B^{1/2})$. Thus the total number of cuts required to perform the fifth cutting is $O(n|\mathcal{Z}|)=O(nB^{1/2}).$ 

Altogether, the number of cuts required for the five types of cuttings is $O_{\eps}(n^{4/3+\eps} + nB^{1/2+\eps}  )$. It remains to verify that after these cuts have been performed, the resulting collection of curves are pseudo-segments. Let $C,C^\prime\in\mathcal{C}$ be circles that intersect at two distinct points $x$ and $y$. Both of the longer arcs from $C$ and $C^\prime$ containing $x$ and $y$ were cut at the first cutting. 

If the shorter arc of $C$ has length at least $\frac{1}{100}$ the the circumference of $C$, then it was cut at the first cutting. Similarly for $C^\prime$. Otherwise, let $v\in\mathcal{V}$ be the orthogonal transformation associated to the pair $(C,C^\prime)$. If  the projection $\pi(O_v(\tilde C))$ contains a point of vertical tangency, then it was cut at the second cutting. Similarly for  $\pi(O_v(\tilde C^\prime))$. Otherwise, the lifts $\beta(\tilde C)$ and $\beta(\tilde C^{\prime})$ form a depth cycle. 

If neither $C$ nor $C^\prime$ are contained in a surface from $\mathcal{Z}_v$, then at least one the arcs $\tilde C$ or $\tilde C^{\prime}$ was cut at the third cutting. If $C$ and $C^\prime$ are contained in a common surface from $\mathcal{Z}$, then at least one of the shorter arcs was cut at the fourth cutting. Finally, if at least one of $C$ or $C^\prime$ is contained in a surface from $\mathcal{Z}$, but the two circles are not contained in a common surface, then at least one of the arcs $\tilde C$ or $\tilde C^{\prime}$ was cut in the firth cutting. 
\end{proof}

\section{Point-Pseudosegment incidences in $\RR^3$}
In \cite{GK}, Guth and Katz proved the following incidence theorem for points and lines in $\RR^3$.
\begin{thm}\label{GKTheorem}
Let $\pts$ be a set of $m$ points and let $\mathcal{L}$ be a set of $n$ lines in $\RR^3$. Suppose that at most $B$ lines can lie in a common plane. Then the number of point-line incidences is 
$$
O\big(m^{1/2}n^{3/4}+m^{2/3}n^{1/3}B^{1/3}+m+n\big).
$$
\end{thm}
\begin{rem}
Guth and Katz actually proved a more general theorem about the number of $k$--rich points spanned by a set of lines in $\RR^3$. However, Theorem \ref{GKTheorem} is an immediate consequence of their result. The proof of Theorem \ref{GKTheorem} is significantly easier than the proof of Guth and Katz's full result.
\end{rem}

We will need a slight variant of Theorem \ref{GKTheorem}, which we will state and prove below.

\begin{lem}\label{pointPseudosegmentIncidencesLem}
Let $\pts$ be a set of $m$ points and let $\mathcal{C}$ be a set of $n$ space curves in $\RR^3$, each of degree at most $t$. Suppose that the curves in $\mathcal{C}$ can be cut into $N$ pseudo-segments. Suppose furthermore that at most $B$ curves from $\mathcal{C}$ can be contained in any surface of degree $\leq100 t^2$. Then the number of point-curve incidences is
\begin{equation*}
O_t \big(m^{1/2}n^{3/4}+m^{2/3}n^{1/3}B^{1/3} + m + N\big).
\end{equation*}

\end{lem}

\begin{rem}
By Lemma \ref{circlesPlanesSpheres}, if the curves in $\mathcal{C}$ are circles, then the requirement that at most $B$ curves are contained in any surface of degree $\leq100t^2$ can be replaced by the requirement that at most $B$ curves are contained in any plane or sphere.
\end{rem}

Before proving Lemma \ref{pointPseudosegmentIncidencesLem}, we will need the following result, which can also be obtained by slightly modifying the proof in \cite[Lemma 3.1]{SZ}.

\begin{lem}\label{incidencesOnASurfae}
Let $\pts$ be a set of $m$ points and let $\mathcal{C}$ be a set of $n$ space curves in $\RR^3$, each of degree at most $t$. Suppose that the points and curves are contained in a variety $Z(P)$, with $0<\deg(P) \leq D$. Suppose furthermore that the curves in $\mathcal{C}$ can be cut into $N$ pseudo-segments.  Then the number of point-curve incidences is
$$
O_{t,D}(m^{2/3}n^{2/3}+m+N).
$$
\end{lem}
\begin{proof}
Since the proof of Lemma \ref{incidencesOnASurfae} is fairly standard, we will just give a brief sketch. First, note that the points contained in the singular locus of $Z(P)$ can contribute $O_{D,t}(m+n)=O_{D,t}(m+N)$ incidences. Thus it suffices to bound the number of incidences between curves and points not contained in the singular locus of $Z(P)$; to simplify notation we will assume that none of the points in $\pts$ are contained in the singular locus of $Z(P)$. 

After applying a rotation if necessary, we can assume that no points of $\pts$ are contained in $Z(\partial_{x_3}(P))$, and that the image of each curve in $\mathcal{C}$ under the projection $\pi(x_1,x_2,x_3)=(x_1,x_2)$ is an algebraic curve. Let $W_1,\ldots,W_s$, $s=O_{D}(1)$ be the set of connected components of $Z(P)\backslash Z(\partial_{x_3}(P))$. The projection $\pi$ is injective on each set $W_i$. Let $\pts_i=\pi(\pts\cap W_i)$, and let $\Gamma_i=\{\pi(C\cap W_i)\colon C\in\mathcal{C}\}$. Then the curve segments in $\Gamma_i$ can be cut into $N$ pseudo-segments. Applying \cite[Lemma 3.1]{SZ}, we conclude that for each index $i$, there are $O_{t}(m^{2/3}n^{2/3}+N)$ point-curve incidences between the curves in $\mathcal{C}$ and the points in $\pts_i$. The lemma now follows from the observation that there are $O_{D}$ indices in total.
\end{proof}

We can now begin the proof of Lemma \ref{pointPseudosegmentIncidencesLem}.

\begin{proof}[\textbf{Proof of Lemma \ref{pointPseudosegmentIncidencesLem}}]
We will prove that if $A$ is sufficiently large, depending only on $t$, then the number of point-curve incidences is at most
$$
A\big(m^{1/2}n^{3/4}+m^{2/3}n^{1/3}B^{1/3} + m + N\big).
$$

First, we can assume $m\leq n^{3/2}$, since otherwise the result follows from \cite[Theorem 1.2]{GZ}. We will prove the result by induction on $n$. However, we are doing this merely for convenience to simplify the discussion of a few non-critical terms. In particular, our use of induction does not introduce an $\eps$ loss in the exponent\footnote{However, since we introduce an $\eps$ loss in the exponent elsewhere in the argument, doing so here as well would not affect the final theorem.}. 

Recall that the curves in $\mathcal{C}$ can be cut into $N$ pseudo-segments. For each $C\in\mathcal{C}$, let $N(C)$ be the number of Jordan arcs in this cutting that are contained in $C$. In particular, $N=\sum_{C\in\mathcal{C}}N(C)$. If $\mathcal{C}^\prime\subset\mathcal{C}$, define $N(\mathcal{C}^\prime)=\sum_{C\in\mathcal{C}^\prime}N(C)$.

Let 
$$
D=c m^{1/2}n^{-1/4},
$$
where $c>0$ is a small constant to be determined later. Note that since $m\leq n^{3/2}$, we have $D\leq cn^{1/2}$.

Use Theorem \ref{projectedPartitioning} (with $d=3,\ e=0,\ f=3$) to find a polynomial $P$ of degree at most $D$ that partitions $\RR^3$ into $O(D^3)$ cells, each of which contains $O(mD^{-3})$ points from $\pts$.

For each cell $\omega$, let $m_\omega$ be the number of points from $\pts$ contained in $\omega$, and let $n_\omega$ be the number of curves that intersect $\omega$. We have 
$$
\sum_\omega m_\omega = |\pts\backslash Z(P)|
$$ 
and 
\begin{equation}
\sum_{\omega}n_\omega\lesssim_t D|\mathcal{C}_0| + N(\mathcal{C}_0),
\end{equation}
where $\mathcal{C}_0\subset\mathcal{C}$ is the set of curves that are not contained in $Z(P)$.

The incidence contribution from the cell interiors is
\begin{equation*}\label{inCellIncidences}
\begin{split}
I(\pts\backslash Z(P),\ \mathcal{C}) &\lesssim  \sum_{\omega} m_\omega n_\omega^{1/2}+ n_\omega \\
&\leq \Big(\sum_{\omega} m_\omega^2\Big)^{1/2} \Big(\sum_{\omega} n_\omega\Big)^{1/2}+ \sum_{\omega}n_\omega \\
& \lesssim_t m|\mathcal{C}_0|^{1/2}D^{-1}+m N(\mathcal{C}_0)^{1/2} D^{-3/2}+D|\mathcal{C}_0|+N(\mathcal{C}_0)\\
&\lesssim_t m^{1/2}n^{3/4}+m^{1/4}n^{3/8}N(\mathcal{C}_0)^{1/2}  +N(\mathcal{C}_0)\\
&\lesssim_t m^{1/2}n^{3/4} +N(\mathcal{C}_0),
\end{split}
\end{equation*}
where on the second line we used the Cauchy-Schwarz inequality and on the final line we used the AM-GM inequality.
Thus 
\begin{equation}\label{inCellIncidences}
I(\pts\backslash Z(P),\mathcal{C})\leq \frac{A}{3}\big(m^{1/2}n^{3/4}+N(\mathcal{C}_0)\big),
\end{equation}
provided $A$ is chosen sufficiently large (depending only on $t$).

Factor $P$ into its irreducible components. Let 
$$
\mathcal{C}_1 = \{C\in\mathcal{C} \colon C\subset Z(P)_{\operatorname{sing}}\}.
$$
We have $|\mathcal{C}_1|\leq D^2\leq cn$. We will choose $c$ sufficiently small so that $|\mathcal{C}_1|\leq n/200$. For each irreducible component $Q$ of $P$, let 
$$
\mathcal{C}_{Q}=\{C\in\mathcal{C}\backslash\mathcal{C}_1\colon C\subset Z(Q)\}.
$$ 
Note that the sets $\{\mathcal{C}_{Q}\}$ are disjoint. For each irreducible component $Q$ of $P$, define
$$
\pts_{Q}=\{p\in\pts\cap Z(Q) \backslash Z(P)_{\operatorname{sing}}\}.
$$
Again, the sets $\{\pts_{Q}\}$ are disjoint.  Define
$$
\pts_{Q}^\prime=\pts_{Q}\cap\pts_2(\mathcal{C}_{Q})
$$
Note that 
$$
\sum_{Q}I(\pts_{Q}\backslash\pts_{Q}^\prime,\mathcal{C}_{Q})\leq |\pts\cap Z(P)|.
$$

Define 
$$
\mathcal{C}_{Q}^\prime=\{C\in\mathcal{C}_{Q}\colon C\cap\pts_Q^\prime\geq A_1D\}, 
$$
where $A_1$ is a constant to be determined later. Note that
$$
\sum_{Q}I(\pts_Q,\mathcal{C}_{Q}\backslash\mathcal{C}_Q^\prime )\leq A_1Dn.
$$

Define

\begin{equation*}
\begin{split}
\mathcal{Q}_1&=\{Q\ \textrm{an irreducible component of}\ P,\ |\mathcal{C}_Q^\prime| < A_2(\deg Q)^2\}\\
\mathcal{Q}_2&=\{Q\ \textrm{an irreducible component of}\ P,\ |\mathcal{C}_Q^\prime|\geq A_2(\deg Q)^2\},
\end{split}
\end{equation*}
where $A_2$ is a constant that will be chosen later. It will depend only on $t$.

Define
$$
\mathcal{C}_2=\bigcup_{Q\in\mathcal{Q}_1}\mathcal{C}_Q^\prime.
$$
We have
$$
|\mathcal{C}_2|\leq\sum_{Q\in\mathcal{Q}_1}A_2 (\deg Q)^2 \leq A_2 (\deg P)^2 \leq A_2D^2\leq cA_2n.
$$

If the constant $c$ is chosen sufficiently small (compared to $A_2$), then $|\mathcal{C}_2|\leq n/200.$  Furthermore, 
$$
I(\pts,\mathcal{C}_1\cup\mathcal{C}_2) = I\big(\pts_1,\mathcal{C}_1\cup\mathcal{C}_2\big),
$$
where 
$$
\pts_1=\pts\cap\Big(Z(P)_{\operatorname{sing}}\cup\bigcup_{Q\in\mathcal{Q}_1}Z(Q)\Big).
$$

Apply the induction hypothesis (with the same value of $B$) to $\pts_1$ and $\mathcal{C}_1\cup\mathcal{C}_2$ and conclude that
\begin{equation}\label{Q1Incidences}
\begin{split}
I(\pts,\mathcal{C}_1\cup\mathcal{C}_2)&\leq A \Big(|\pts_1|^{1/2}(n/100)^{3/4}+|\pts_1|^{2/3}(n/100)^{1/3}B^{1/3} + |\pts_1| + N(\mathcal{C}_1\cup\mathcal{C}_2)\Big),\\
&\leq A\Big(\frac{1}{10}(m^{1/2}n^{3/4}+m^{2/3}n^{1/3}B^{1/3})+|\pts_1|+N(\mathcal{C}_1\cup\mathcal{C}_2)\Big).
\end{split}
\end{equation}

By Lemma \ref{manyCurvesInSurface}, if $A_1$ is chosen sufficiently large then for each $Q\in\mathcal{Q}_2$, $Z(Q)$ is doubly-ruled by curves of degree $\leq t$. In particular, $\deg(Q)\leq100t^2$. 

We can now analyze each of these surfaces separately. Define $\mathcal{C}_3=\mathcal{C}\backslash(\mathcal{C}_0\cup\mathcal{C}_1\cup\mathcal{C}_2)$. By Lemma \ref{incidencesOnASurfae}, we have 
\begin{equation*}
\begin{split}
I(\pts,\mathcal{C}_3)&\lesssim \sum_{Q\in\mathcal{Q}_2} \Big(|\pts_{Q}^\prime|^{2/3}|\mathcal{C}_{Q}^\prime|^{2/3} + |\pts_Q^\prime| + N(\mathcal{C}_Q^\prime)\Big)\\
&\lesssim \Big(\sum_{Q\in\mathcal{Q}_2} |\pts_{Q}^\prime|\Big(\sum_{Q\in\mathcal{Q}_2}|\mathcal{C}_{Q}^\prime|^2\Big)^{1/3} + \sum_{Q\in\mathcal{Q}_2} |\pts_Q^\prime| + \sum_{Q\in\mathcal{Q}_2} N(\mathcal{C}_Q^\prime)\\
&\lesssim m^{2/3}(nB)^{1/3}+ \sum_{Q\in\mathcal{Q}_2} |\pts_Q| +N(\mathcal{C}_3).
\end{split}
\end{equation*}

Thus if $A$ is chosen sufficiently  large (depending only on $t$), then
\begin{equation}\label{Q2Incidences}
I(\pts ,\mathcal{C}_3)\leq \frac{A}{3} \Big(m_3^{2/3}n^{1/3}B^{1/3}+\Big|\pts\cap \bigcup_{Q\in\mathcal{Q}_2}Z(Q) \backslash Z(P)_{\operatorname{sing}}\Big|+N(\mathcal{C}_3)\Big).
\end{equation}

Combining \eqref{inCellIncidences}, \eqref{Q1Incidences}, and \eqref{Q2Incidences}, we obtain
\begin{equation}
\begin{split}
I(\pts,\Gamma)&\leq A\Big(m^{1/4}n^{3/4}+m^{2/3}n^{1/3}B^{1/3} + \big|\pts\cap\big(Z(P)_{\operatorname{sing}}\cup\bigcup_{Q\in\mathcal{Q}_1} Z(Q)\big)\big|\\
&\qquad+\big|\pts\backslash Z(P)_{\operatorname{sing}}\cap\bigcup_{Q\in\mathcal{Q}_2}Z(Q)\big| + N(\mathcal{C}_0)+N(\mathcal{C}_1\cup\mathcal{C}_2)+N(\mathcal{C}_3)\Big)\\
&\leq A\big(m^{1/2}n^{3/4}+m^{2/3}n^{1/3}B^{1/3}+m+N\big).\qedhere
\end{split}
\end{equation}
\end{proof}
Though we will not need it here, we remark that  Lemma \ref{pointPseudosegmentIncidencesLem} plus \cite[Theorem 1.2]{GZ} yields the following corollary. 
\begin{cor}
Let $\mathcal{C}$ be a set of $n$ space curves in $\RR^3$, each of degree at most $t$. Suppose that the curves in $\mathcal{C}$ can be cut into $N$ pseudo-segments. Suppose furthermore that at most $B$ curves from $\mathcal{C}$ can be contained in any surface of degree $\leq100 t^2$. Then the number of $k$-rich points is
$$
O_t(n^{3/2}k^{-2} + nBk^{-3}+Nk^{-1}).
$$
\end{cor}
\begin{proof}
Suppose that there are $m$ $k$--rich points. By \cite[Theorem 1.2]{GZ}, we have $m\leq n^{3/2}$, so in particular, $m=m^{1/2}m^{1/2}\leq m^{1/2}n^{3/4}$. By Lemma \ref{pointPseudosegmentIncidencesLem}, we have
 
$$
mk =O_t\big(m^{1/2}n^{3/4}+m^{2/3}n^{1/3}B^{1/3} + m + N\big) =O_t \big(2m^{1/2}n^{3/4}+m^{2/3}n^{1/3}B^{1/3} + N\big),
$$
which can be re-arranged to yield
$$
m = O_t(n^{3/2}k^{-2} + nBk^{-3}+Nk^{-1}).\qedhere
$$
\end{proof}
\section{Finding structure amongst sets of circles}

In this section we will prove the following result.
\begin{lem}\label{circlePartitioningLem}
Let $0\leq\alpha< 1/2$. For each $\eps>0$, there is a constant $A$ so that the following is true. Let $\pts\subset\RR^3$ be a set of $m$ points and let $\circles$ be a set of $n>4^{1/(1-2\alpha)}$ circles in $\RR^3$, no two of which are co-planar\footnote{see Remark \ref{coplanarRemark}}. 

Then for each $E\geq 1$, at least one of the following three things must hold:\\
\begin{itemize}
\item[(A)] (The circles are contained in a small number of spheres): There is a set $\mathcal{C}_{(A)}\subset \mathcal{C}$ of size at least $n/\log n$, and a set $\mathcal{S}$ of spheres with $|\mathcal{S}|\leq n^\alpha$. For each $S\in\mathcal{S}$, there is a set $\circles(S)\subset\circles_{(A)}$ of circles contained in $S$ so that the sets $\{\circles(S)\}_{S\in\mathcal{S}}$ are disjoint and $|\circles(S)|\leq 2n|\mathcal{S}|^{-1}$ for each $S\in\mathcal{S}$. Finally,
\begin{equation}\label{incidencesSumOverSpheres}
I(\pts,\circles_{(A)})=  \sum_{S\in\mathcal{S}}I(\pts, \circles(S)).
\end{equation}

\item[(B)] (The circles are contained in an intermediate number of spheres): There is a set $\mathcal{C}_{(B)}\subset \mathcal{C}$ of size at least 
\begin{equation}\label{sizeOfCB}
|\mathcal{C}_{(B)}|\geq E^{-\eps}n-E^{6+2\eps},
\end{equation}
a number $1\leq F<E$, and a set $\Omega$ of at most $AF^{6+\eps}$ ``cells.'' For each cell $\omega\in\Omega$, there is a set $\pts_\omega\subset\pts$ of points, a set $\circles_\omega\subset\circles_{(B)}$ of circles with $|\circles_\omega|\leq AF^{-6+\eps}n$, and a set $\mathcal{S}_\omega$ of spheres with $|\mathcal{S}_\omega|\leq |\circles_\omega|^{\alpha}$. For each $S\in\mathcal{S}_\omega$, there is a set $\circles(S)\subset\circles_\omega$ of circles contained in $S$.

The sets $\{\circles_\omega\}_{\omega\in\Omega}$ are disjoint and for each $\omega\in\Omega$, the sets $\{\circles(S)\}_{S\in\mathcal{S}_\omega}$ are disjoint. The sets $\{\pts_\omega\}_{\omega\in\Omega}$ are not disjoint, but 
$$
\sum_{\omega\in\Omega} |\pts_\omega|\leq A F^{4+\eps}m.
$$ 
For each $\omega\in\Omega$ and each $S\in \mathcal{S}_\omega,$ we have 
$
|\circles(S)|\leq A |\circles_\omega|^{1-\alpha}.
$
Finally,
\begin{equation}\label{incidencesSumOverCellsSpheres}
I(\pts,\circles_{(B)})\leq \sum_{\omega\in\Omega}\sum_{S\in\mathcal{S}_\omega}I(\pts_\omega,\circles(S)) + AE^{\eps}n.
\end{equation}

\item[(C)] (The circles strongly avoid being contained in spheres): There is a set $\mathcal{C}_{(C)}\subset \mathcal{C}$ of size at least $E^{-\eps}n-AE^{6+2\eps}$, and a set $\Omega$ of at most $AE^{6+\eps}$ ``cells.'' For each cell $\omega\in\Omega$ there is a set $\pts_\omega$ of points and a set $\circles_\omega$ of circles. The sets $\{\circles_\omega\}$ are disjoint, and $|\circles_\omega|\leq AE^{-6+\eps}n$ for each $\omega\in\Omega$. The sets $\{\pts_\omega\}$ are not disjoint, but $\sum_\Omega |\pts_\omega|\leq AE^{4+\eps}m$. For each $\omega\in \Omega$, at most $|\circles_{\omega}|^{1-\alpha}$ circles from $\circles_\omega$ can be contained in a common sphere. We also have

\begin{equation}\label{incidencesSumNoSpheres}
I(\pts,\circles)\leq \sum_{\omega\in\Omega}I(\pts_\omega,\circles_\omega)+ AE^{\eps}n.
\end{equation}
\end{itemize}
\end{lem}
\begin{rem}\label{coplanarRemark}
The requirement that no two circles are co-planar is harmless, since we can always ensure that this holds by applying an inversion around a suitable point in $\RR^3$. Equivalently, we could remove this requirement and replace every reference to ``spheres'' with ``spheres and planes.''
\end{rem}
The proof of this lemma is a type of ``stopping time'' argument. We begin by checking whether the circles in $\circles$ cluster into a small number of spheres. If so, we are in Case (A). If not, we examine the points and four-dimensional varieties in $\RR^6$ that are dual to the circles from $\circles$ and points from $\pts$. We partition these points and varieties into a small number of cells, and examine the behavior inside these cells. If the circles inside most cells now cluster into a small number of spheres, then we are in Case (B). If not, we iteratively continue this partitioning process inside each cell. If at some point the circles inside most cells cluster into a small number of spheres then we are in Case (B). If we succeed in partitioning the circles into $E^6$ cells without Case (B) occurring, then we are in Case (C). 

Before beginning the process described above, we must develop a partitioning theorem for the points and varieties in $\RR^6$ that are dual to circles and points in $\RR^3$. The main technical difficulty is that circles in $\RR^3$ correspond to points in $\RR^6$, while points in $\RR^3$ correspond to four-dimensional varieties in $\RR^6$. These four-dimensional varieties are of a special type, and the details of how these varieties can intersect will be exploited to yield a stronger partitioning result than is possible for more general classes of four-dimensional varieties.

\subsection{The dual space of points and circles}\label{dualSpacePtsCirclesSec}
For each $p\in\RR^3$, let $S_p=\{q\in\RR^3\colon |p-q|=1\}$; this is the unit sphere centered at $p$. We will identify the parameter space of circles in $\RR^3$ of radius $0<r<1$ with a subset of $\RR^6$ in the following manner: identify each point $(p,p^\prime)\in\RR^3\times\RR^3$ with the circle $C_{pp^\prime}=S_p\cap S_{p^\prime}$. If $0<|p-p^\prime|<1$ then this is a circle. If $|p-p^\prime|=1$ then this is a point, while if $|p-p^\prime|>1$ then this set is empty. For each $p\in\RR^3$, define $Z_p = S_p\times S_p\subset\RR^6$.  We will call sets of the form $S_p\times S_p$ ``double-spheres.''

We will make use of the following duality: If $C_{p,p^\prime}\subset\RR^3$ is a circle and if $q\in C_{p,p^\prime}$, then $(p,p^\prime)\in Z_q$. The following lemma records the fact that at most two unit spheres can be mutually tangent at a common point.
\begin{lem}\label{vectorSpaceSumLem}
Let $q_1,q_2,q_3\in\RR^3$ and let $w\in\RR^6$. Suppose that $w\in Z_{q_i},\ i=1,2,3.$ Then $T_wZ_{q_1}+T_wZ_{q_2}+T_wZ_{q_3}=\RR^6$.
\end{lem}

Note that if $q_1,q_2\in\RR^3$ are distinct points, then $Z_{q_1}\cap Z_{q_2}=(S_{q_1}\cap S_{q_2})\times (S_{q_1}\cap S_{q_2})$. If this set is non-empty, then it either consists of a single point or is a product of two circles, each embedded in $\RR^3$. 

Furthermore, we have the following.
\begin{lem}\label{evenDim}
Let $q_1,q_2\in\RR^3$ and let $w\in Z_{q_1}\cap Z_{q_2}$. Then $\dim(T_wZ_{q_1}\cap T_wZ_{q_2})$ must be even. 
\end{lem}

The next sequence of lemmas will describe how double-spheres can interact with algebraic varieties of various dimension. 
\begin{lem}\label{fiveDimVariety}
Let $W\subset\RR^6$ be a five-dimensional variety and let $w\in W$ be a regular point of $W$ (in particular, this means that in a neighborhood of $w$, $W$ is a smooth five-dimensional manifold). Let $q_1,q_2,q_3\in\RR^3$ be distinct points, and suppose $w\in Z_{q_i},\ i=1,2,3$. Then $\dim(T_w Z_{q_i} \cap T_w W)\leq 3$ for at least one index $i$.
\end{lem}
\begin{proof}
This is a variant of the tangent space argument from \cite{ST}. Suppose that $\dim(T_w Z_{q_i} \cap T_w W)=4$ for $i=1,2,3$. Then for each index $i$, $T_w Z_{q_i}\subset T_wW$, so the vector-space sum $T_w Z_{q_1}+T_w Z_{q_2}+T_w Z_{q_3}\subset T_wW$. But this contradicts the fact that by Lemma \ref{vectorSpaceSumLem},  $T_w Z_{q_1}+T_w Z_{q_2}+T_w Z_{q_3}=\RR^6$.
\end{proof}

A similar argument establishes the following.
\begin{lem}\label{threeDimVariety}
Let $W\subset\RR^6$ be a three-dimensional variety, and let $w\in W$ be a regular point of $W$. Let $q_1,q_2,q_3\in\RR^3$ be distinct points, and suppose $w\in Z_{q_i},\ i=1,2,3.$ Then $\dim(T_w Z_{q_i} \cap T_w W)\leq 2$ for at least one index $i$.
\end{lem}

If $W$ has dimension one, two, or four, then a more complicated argument is required.
 
\begin{lem}\label{fourDimVariety}
Let $W\subset\RR^6$ be a four-dimensional variety defined by polynomials of degree at most $t$. Then there is a set $W^\prime\subset W$ of dimension at most three, and for each double-sphere $Z_{q}$ there is a variety $Z_{q}^\prime\subset W\cap Z_q$ of dimension at most two so that: if $w\in W\backslash W^\prime$, then there can be at most two double-spheres $Z_q$ with $w\in Z_{q}\backslash Z_{q}^\prime$ so that $w$ is contained in a three-dimensional component of $W\cap Z_q$.  The varieties $W^\prime$ and $Z_q^\prime$ are defined by polynomials whose degree is bounded by a function depending only on $t$.
\end{lem}
\begin{proof}
Let $f_1,\ldots,f_{\ell}$ be generators of $I(W)$. Let 
\begin{equation*}
\begin{split}
\Gamma_W&=\{(w,v,v)\in \RR^6\times \RR^3\times\RR^3\colon w\in W_{\operatorname{reg}},\ v\cdot v=1,\\
&\qquad\qquad (v,0)\cdot \nabla f_i(w)=0\ \textrm{and}\ (0,v)\cdot \nabla f_i(w)=0,\ i=1,\ldots,\ell\}\\
&=\{(w,v,v)\in \RR^6\times \RR^3\times\RR^3\colon w\in W_{\operatorname{reg}},\ v\cdot v=1,\ (v,0)\in T_wW,\ (0,v)\in T_wW\}.
\end{split}
\end{equation*}
The definition of $\Gamma_W$ is motivated by the observation that if $w\in Z_{q_1}\cap Z_{q_2}$ and if $T_wZ_{q_1}\neq T_wZ_{q_2}$, then $T_wZ_{q_1}\cap T_wZ_{q_2}$ contains exactly two vectors of the form $(v,v)\in\RR^3\times\RR^3$ with $|v|=1$. Thus if $w$ is a regular point of $W$ and if $\dim (T_wW\cap T_wZ_{q_i})=3,\ i=1,2,$ then $T_wW$ must contain at least two vectors of the form $(v,v)\in\RR^3\times\RR^3$ with $|v|=1$. The set $\Gamma_W$ will help us measure where this can occur. 

For each $w\in W_{\operatorname{reg}}$, define $\Gamma_W(w)= \Gamma_W\cap (\{w\}\times\RR^6)$. This set is either empty, consists of two unit vectors of the form $\pm v$, or is of the form $\{(v,v)\in\RR^3\times\RR^3\colon v\in S^1\},$ where $S^1\subset\RR^3$ is a circle. 

Let 
$$
W_1=\{w\in W_{\operatorname{reg}}\colon |\Gamma_W(w)|=2)\}
$$ 
(i.e.~$W_1$ is the set of points where $\Gamma_W(w)$ consists of two points) and let $\Gamma_W^\prime=\Gamma_W\cap (W_1\times\RR^6)$. We claim that $\Gamma_W^\prime$ is a semi-algebraic set. To see this, note that $\Gamma_W$ is a semi-algebraic set, so it suffices to prove that $W_1$ is semi-algebraic. Define $W_{\operatorname{circle}}$ to be the projection of the set
$$
\{(w, v_1,v_2, v_3)\in W_{\operatorname{reg}}\times \RR^3\times\RR^3\times\RR^3\colon v_i \neq \pm v_j\ \textrm{if}\ i\neq j; (w,v_i,v_i)\in \Gamma_W\ i=1,2,3\}
$$
to $W$. Define $W_{\operatorname{non-empty}}$ to be the projection of $\Gamma_W$ to $W$ and define $W_{\operatorname{empty}}=W_{\operatorname{reg}}\backslash W_{\operatorname{non-empty}}$. Then 
$$
W_1 = W_{\operatorname{reg}}\backslash\big(W_{\operatorname{empty}}\cup W_{\operatorname{circle}}\big).
$$
Now that we have established that $\Gamma_W$  is semi-algebraic, we can continue with the proof. 

First we will deal with the case where $\dim(\Gamma_W^\prime)\leq 3$. Let $W^\prime\subset W$ be an algebraic variety that contains the union of $W_{\operatorname{sing}}$ and the projection of $\Gamma_W^\prime$. For each double-sphere $Z_q$, let $Z_q^\prime=\emptyset$. To see that this choice of $W^\prime$ and $Z_q^\prime$ satisfies the conclusions of the lemma, suppose there exists $w\in W\backslash W^\prime$ and three double-spheres $Z_{q_1},Z_{q_2},Z_{q_3}$ so that for each $i=1,2,3$, $w\in Z_{q_i}\backslash Z_{q_i}^\prime$ and $w$ is contained in a component of $W\cap Z$ of dimension at least three. In particular, this means that $\Gamma_W(x)$ is non-empty, so $\Gamma_W(x)$ must be of the form $\{(v,v)\in\RR^3\times\RR^3\colon v\in S^1\}.$ This implies that then $T_w(W)$ is a Cartesian product $V\times V$, where $V\subset\RR^3$ is a two-dimensional subspace. Since $Z_{q_i}$ is a Cartesian product $S_{q_i}\times S_{q_i}$, $i=1,2,3$, this means that $\dim(T_wW \cap T_w Z_{q_i})=4,\ i=1,2,3$, which is impossible.

Next we will deal with the case where $\dim(\Gamma_W^\prime)=4$. Recall from Section \ref{realAlgVarietySec} that we can write $\Gamma_W^\prime \subset \Gamma_1\cup\Gamma_2$, where $\Gamma_1$ is a four-dimensional smooth manifold and $\Gamma_2$ is an algebraic variety of dimension $\leq 3$. Define $W^\prime\subset W$ to be a variety of dimension at most three that contains the projection of $\Gamma_2$.
Observe that the projection of $\Gamma_1$ to $W$ is still a smooth manifold; call this set $W_2$. For each point $w\in W_2$, we can associate a unit vector $v(w)\in\RR^3$ so that for each $w\in W_2,$ $(w,v(w),v(w))\in \Gamma_1$ and $(w,-v(w),-v(w))\in \Gamma_1$. Since $\Gamma_1$ is smooth, so is the vector field $(v(w),v(w))$. In particular, every point on $W_2$ lies on a unique integral curve of the vector field $(v(w),v(w))$ (see e.g. \cite[Chapter 9]{Lee} for additional background on vector fields and their associated integral curves). 

For each double-sphere $Z_q$, define 
$$
\Gamma_{W\cap Z_q}=\{(w,v,v)\in (W\cap Z_q)_{\operatorname{reg}}\colon v\cdot v=1,\ (v,0)\in T_w( W\cap Z_q),\ (0,v)\in T_w( W\cap Z_q)\}.
$$
As above, define 
$$
Z_{q,1}=\{w\in (W\cap Z_q)_{\operatorname{reg}}\colon |\Gamma_{W\cap Z_q}(w)|=2)\}.
$$ 
Define $\Gamma_{W\cap Z_q}^\prime= \Gamma_{W\cap Z_q}\cap (Z_{q,1}\times\RR^6)$. If $\dim(\Gamma_{W\cap Z_q}^\prime)=2$, let $Z_q^\prime$ be the projection of $\Gamma_{W\cap Z_q}^\prime$ to $W\cap Z_q$. If instead $Z_{q,1}$ has dimension three, write $Z_{q,1}\subset Z_{q,2}\cup Z_q^\prime,$ where $Z_{q,2}$ is a three dimensional manifold and $Z_q^\prime$ is an algebraic variety of dimension at most two. As discussed above, we obtain a smooth vector field $(v(w),v(w))$ on $Z_{q,2}$. Note that if $w\in W_2\cap Z_{q,2}$, then the vector fields $(v(w),v(w))$ arising from $W_2$ and from $Z_{q,1}$ agree. Thus their integral curves also agree in a small (Euclidean) neighborhood of $w$.

We will show that this choice of $W^\prime$ and $Z_q^\prime$ satisfies the conclusions of the lemma. Indeed, suppose there exists $w\in W\backslash W^\prime$ and there are three double-spheres $Z_{q_1},Z_{q_2},Z_{q_3}$ so that for each $i=1,2,3$, $w\in Z_{q_i}\backslash Z_{q_i}^\prime$ and $w$ is contained in a component of $W\cap Z$ of dimension at least three. In particular, this means that $\Gamma_W(x)$ is non-empty.

If $\Gamma_W(w)$ is of the form $\{(v,v)\colon v\in S^1\}$, then as noted above, $T_w(W)$ is a Cartesian product $V\times V$, where $V\subset\RR^3$ is a two-dimensional subspace. Since $Z_{q_i}$ is a Cartesian product $S_{q_i}\times S_{q_i}$, $i=1,2,3$, this means that $\dim(T_wW \cap T_w Z_{q_i})=4,\ i=1,2,3$, which is impossible.

Now suppose $|\Gamma_W(x)|=2$. Then $w\in W_1$. Since by assumption $w\not\in W^\prime$, we must have $w\in W_2$. Then there is a Euclidean neighborhood $O$ of $w$ and an integral curve $\gamma\subset W_2\cap O$ of the vector field $(v(w),v(w))$ that contains $w$ so that $\gamma$ is contained in each of $Z_{q_1},Z_{q_2}$, and $Z_{q_3}$. But this is impossible, since every triple intersection of double-spheres must consist of at most four points. 
\end{proof}

A similar (but easier) argument establishes the following analogous result for two-dimensional varieties and one dimensional varieties:
\begin{lem}\label{twoDimVariety}
Let $W\subset\RR^6$ be a variety of dimension one or two defined by polynomials of degree at most $t$. Then there is a set $W^\prime\subset W$ of dimension at most $\dim(W)-1$, and for each double-sphere $Z_{q}$ there is a variety $Z_{q}^\prime\subset W\cap Z_q$ of dimension at most $\dim(W)-1$, so that if $w\in W\backslash W^\prime$ then there can be at most two double-spheres $Z_q$ with $w\in Z_{q}\backslash Z_{q}^\prime$ so that $w$ is contained in a component of $W\cap Z_q$ of dimension $\geq \dim(W)-1$. The varieties $W^\prime$ and $Z_q^\prime$ are defined by polynomials whose degree is bounded by a function that depends only on $t$ (if the varieties are zero-dimensional, this means that they are finite sets whose cardinality is bounded by a function that depends only on $t$).
\end{lem}


Combining the above results, we obtain the following.

\begin{lem}\label{twoThirdsDimensionLem}
Let $W\subset\RR^6$ be a real algebraic variety of dimension at least one defined by polynomials of degree at most $t$. Let $\pts\subset W$ be a set of points and let $\mathcal{Z}$ be a set of double-spheres. Then there is a set $W^\prime\subset W$; for each $Z\in\mathcal{Z}$, there is a set $Z^\prime\subset Z$; and there is a set $I_0\subset \mathcal{I}(\pts,\mathcal{Z})$ with the following properties:
\begin{itemize}
\item $W^\prime$ is a variety of dimension strictly less than $\dim(W)$, defined by polynomials whose degree is bounded by a function depending only on $t$.
\item $|I_0|\leq 2|\pts\backslash W^\prime|$.
\item $\dim(Z^\prime)\leq \frac{2}{3}\dim(W)$ for each $Z\in\mathcal{Z}$.
\item If $p\in \pts\backslash W^\prime$, $Z\in\mathcal{Z}$, and $p\in Z\backslash Z^\prime$, then either $(p,Z)\in I_0$ or the irreducible component of $W\cap Z$ containing $p$ has dimension at most $\frac{2}{3}\dim(W)$.
\end{itemize}
\end{lem}
\begin{proof}
If $W=\RR^6$, let $W^\prime=\emptyset$ and for each $Z\in\mathcal{Z}$, let $Z^\prime=\emptyset$. Let $I_0=\emptyset$. We can immediately verify that This choice of $W^\prime, Z^\prime$ and $I_0$ satisfy the conclusions of the lemma. 

If $\dim(W)=3$ or $5$, let $W^\prime = W_{\operatorname{sing}}$. For each $Z\in\mathcal{Z}$, let $Z^\prime=(Z\cap W)_{\operatorname{sing}}$. Let 
\begin{equation}\label{defnOfI0}
I_0=\{(p,Z)\in\pts\times\mathcal{Z}\colon p\in W_{\operatorname{reg}},\ p\in (Z\cap W)_{\operatorname{reg}},\ \dim(T_p(Z\cap W)>\frac{2}{3}\dim(W)\}. 
\end{equation}
By Lemmas \ref{fiveDimVariety} and \ref{threeDimVariety}, we have that for each $p\in \pts\cap W_{\operatorname{reg}}$, there are at most two double-spheres $Z\in\mathcal{Z}$ with $(p,Z)\in I_0$. We can also verify that if $\dim(W)=5$, then $\dim(Z^\prime)\leq 4-1\leq \frac{2}{3}5$, while if $\dim(W)=3$ then $\dim(Z^\prime)\leq 3-1\leq \frac{2}{3}3.$

Next, if $\dim(W)=1,2,$ or $4$, let $W^\prime$ be the set given by Lemma \ref{twoDimVariety} or \ref{fourDimVariety}, respectively, and for each $Z\in\mathcal{Z}$ let $Z^\prime$ be the set given by Lemma \ref{twoDimVariety} or \ref{fourDimVariety} respectively, and define $I_0$ as in \eqref{defnOfI0}. 
\end{proof}

\subsection{Partitioning points and circles in dual space}
\begin{lem}\label{partitioningAVarietyLem}
For each $t\geq 1$, there is a constant $A$ so that the following holds. Let $W\subset\RR^6$ be a real algebraic variety of dimension at least one that is defined by polynomials of degree at most $t$. Let $\pts\subset W$ be a set of points and let $\mathcal{Z}$ be a set double-spheres. 

Then for each $D\geq 1$, there is:
\begin{itemize} 
\item An algebraic variety $V$ defined by polynomials of degree $\max(A,D)$, with $\dim(V)<\dim(W)$.
\item A set $\Omega$ of $O(D^6)$ cells (i.e.~subsets of $\RR^6$).
\item For each cell $\omega\in\Omega$, a set $\mathcal{Z}_{\omega}\subset\mathcal{Z}$.
\item A set $I_1\subset \mathcal{I}(\pts,\mathcal{Z})$.
\end{itemize}
These objects have the following properties.

The set $I_1$ is small:
$$
|I_1|\leq 2|\pts\backslash V|.
$$

Each cell contains few points. More precisely, for each $\omega\in\Omega$, 
\begin{equation}\label{D6PointsPerCell}
|\pts\cap\omega|\leq A|\pts\backslash V|D^{-6}.
\end{equation}

The varieties enter a controlled number of cells:
\begin{equation}\label{D4SurfacesPerCell}
\sum_{\omega}|\mathcal{Z}_\omega|\leq AD^4|\mathcal{Z}|.
\end{equation}

Finally, the sets $I_1$, $\Omega$ and $\{\mathcal{Z}_{\omega}\}$ capture the incidences between the points and surfaces. Specifically, if $(p,Z)\in \mathcal{I}(\pts,\mathcal{Z})$ then either $p\in V$; $(p,Z)\in I_1$; or $Z\in\mathcal{Z}_\omega$, where $\omega$ is the (unique) cell containing $p$. 
\end{lem}
\begin{rem}
The main point of this lemma is that \eqref{D6PointsPerCell} says that at most a $O(D^{-6})$ fraction of the points are contained in each of the $O(D^6)$ cells, while \eqref{D4SurfacesPerCell} says that on average, at most a $O(D^{-2})$ fraction of the surfaces intersect each cell. This is a better ratio than one would expect if the varieties in $\mathcal{Z}$ were arbitrary four-dimensional varieties.
\end{rem}
\begin{proof}
Apply Lemma \ref{twoThirdsDimensionLem} to $W$, $\pts,$ and $\mathcal{Z}$, and let $I_0$, $W^\prime$, and $\{Z^\prime\}_{Z\in\mathcal{Z}}$ be the output of the lemma. Define $f=\dim(W)$ and for each $Z\in\mathcal{Z}$, define
\begin{equation*}
\begin{split}
\mathcal{Y}_Z= &\{Y\colon\ Y\ \textrm{is an irreducible component of}\ Z\cap W\ \textrm{with}\ 0<\dim(Y)\leq\frac{2}{3} f \}\\
&\cup \{Y\colon\ Y\ \textrm{is an irreducible component of}\ Z^\prime\}.
\end{split}
\end{equation*}

Note $|\mathcal{Y}_Z|\leq A_1$, where $A_1$ is a constant that depends only on $t$. Furthermore, each variety $Y\in\mathcal{Y}_Z$ has dimension at most $\frac{2}{3}f$ and is defined by polynomials of degree at most $A_1$. By Lemma \ref{twoThirdsDimensionLem}, if $(p,Z)\in I(\pts,\mathcal{Z})$ then either $(p,Z)\in I_0$, $p\in W^\prime$, or $p\in Y$ for some variety $Y\in\mathcal{Y}_Z$.

Apply a generic rotation and then apply Theorem \ref{projectedPartitioning} (with $d=6,\ e=0,\ f=\dim W$). Let $P$ be the resulting partitioning polynomial of degree at most $D^{6/f}$; we have at most $A_1D^6$ cells, with at most $A_1|\pts|D^{-6}$ points in each cell. For each $Z\in\mathcal{Z}$, define
$$
\Omega(Z)=\{\omega\colon \omega\cap Y\neq\emptyset\ \textrm{for some}\ Y\in\mathcal{Y}_Z\}.
$$
By Theorem \ref{numberOfConnectedComponents}, each $Y\in\mathcal{Y}_Z$ can intersect 
$$
O\big((\deg P)^{\dim Y}\big)=O\big( (D^{6/f})^{\frac{2}{3}f}\big)=O(D^4)
$$
cells, where the implicit constant depends only on $t$. Thus 
\begin{equation}\label{boundOnOmegaZ}
|\Omega(Z)|\leq AD^4,
\end{equation}
where the constant $A$ depends only on $t$. For each cell $\omega\in\Omega$, define 
$$
\mathcal{Z}_{\omega}=\{Z\in\mathcal{Z}\colon\omega\in\Omega(Z)\}.
$$

By \eqref{boundOnOmegaZ},
$$
\sum_{\omega}|\mathcal{Z}_\omega|\leq A D^4|\mathcal{Z}|.
$$ 

Let $V= W^\prime\cup \big(W \cap Z(P)\big)$. Define $I_1 = \{(p,Z)\in I_0\colon p\in W\backslash V\}$. Since each $p\in\pts$ can occur in $I_0$ at most twice, we have $|I_1|\leq 2|\pts\backslash V|.$
\end{proof}



%
%
The following multi-level partitioning theorem is closely related to the partitioning theorem of Matou\v{s}ek and Pat\'akov\'a \cite{MP}. However, the partitioning theorem in \cite{MP} is for points and hypersurfaces, while the present partitioning theorem exploits the fact that the varieties we wish to partition have the special intersection properties described in Lemma \ref{twoThirdsDimensionLem}. This allows us to obtain a stronger result.

\begin{lem}[Multi-level partitioning for dual circles]\label{partitionCirclesLem} 
For each $\eps>0$, there is a constant $A$ so that the following holds. Let $\pts\subset\RR^6$ be a set of points. Let $\mathcal{Z}$ be a set of double-spheres. Then for each $D\geq 1$, there is a partition of $\RR^6$ into a set $\Omega$ of $\leq A D^{6+\eps}$ cells, a set $I_0\subset \mathcal{I}(\pts,\mathcal{Z})$ of size at most
\begin{equation}\label{boundOnI0}
|I_0|\leq A D^{\eps} |\pts|,
\end{equation}
and a set $\pts^*\subset\pts$ of size at most 
\begin{equation}\label{sizeOfPStar}
|\pts^*|\leq AD^{6+\eps}.
\end{equation}
These sets have the following properties.
\begin{itemize}
\item Each cell $\omega$ contains at most $A|\pts|D^{-6+\eps}$ points from $\pts$.
\item For each cell $\omega$, there is a set $\mathcal{Z}_{\omega}\subset\mathcal{Z}$ so that $\sum |\mathcal{Z}_{\omega}|\leq AD^{4+\eps}|\mathcal{Z}|$.
\item If $(p,Z)\in \mathcal{I}(\pts,\mathcal{Z})$ then either $p\in\pts^*$; $(p,Z)\in I_0$; or $Z\in\mathcal{Z}_\omega$, where $\omega$ is the (unique) cell containing $p$.
\end{itemize}
\end{lem}
\begin{proof}
We will prove the result by induction on $|\pts|$. If $|\pts|$ is sufficiently small then the result is immediate if we choose $A$ large enough. Similarly, observe that if $D$ is sufficiently small then the result is immediate if we choose $A$ large enough; simply let $\Omega=\{\RR^6\}.$ Thus we can assume that $D$ is larger than some constant $D_{\operatorname{min}}$ that depends only on $\eps$.

Let $W_0 = \RR^6$, let $\pts_0 = \pts$, and let $D_0$ be a constant depending only on $\eps$. For each $i=1,\ldots,5$ with $\dim W_i\geq 1$, let $D_i$ be a large constant depending only on $D_0,\ldots,D_{i-1}$ (of course, this means that the constants $D_i$ will depend only on $\eps$). We will select $D_{\operatorname{min}}$ sufficiently large (depending only on $\eps$) so that $D_{\operatorname{min}}/D_i\geq 1$ for each index $i$.

 Apply Lemma \ref{partitioningAVarietyLem} to $W_i$ and $\pts_i\cap (W_i)_{\operatorname{reg}}$ with parameter $D_i$. Let $W_{i+1}$ be the resulting variety of dimension at most $\dim(W_i)-1$, and let $\Omega_i$, $\{\mathcal{Z}_\omega\}_{\omega\in\Omega_i}$ and $J_i$ be the resulting set of cells, varieties, and excess incidences. 

Define $\pts_{i+1} = \pts_i\cap W_{i+1}$ and define $\pts_i^\prime=\pts_i\backslash \pts_{i+1}$. Note that 
$$
|J_i|\leq 2 |\pts_i^\prime|.
$$
There is a number $A_i$, depending only on $D_0,\ldots,D_{i-1}$ so that $|\Omega_i|\leq A_i D_i^6$;  $|\pts_{i}^\prime\cap\omega|\leq A_i D_i^{-6}|\pts_i|$ for each $\omega\in\Omega_i$; and $\sum_{\omega\in\Omega_i}|\mathcal{Z}_\omega|\leq A_iD_i^4|\mathcal{Z}|$.

Apply the induction hypothesis inside each cell $\omega\in\Omega_i$ with parameter $D/D_i$ (since $D\geq D_{\operatorname{min}}$ by assumption, we have $D/D_i\geq 1$ so the induction hypothesis is satisfied). If $D_i$ is chosen sufficiently large compared to $A_i$ (which in turn depends only on $D_0,\ldots D_{i-1}$), then we obtain a set of cells $\Omega_\omega$, and for each $\omega^\prime\in\Omega_{\omega}$, a set $\mathcal{Z}_{\omega^\prime}$ so that 
$$
\sum_{\omega\in\Omega}|\Omega_{\omega}|\leq A_iD_i^6 A(D/D_i)^{6+\eps}\leq \frac{1}{6}AD^{6+\eps};
$$
for each $\omega\in\Omega$, each cell $\omega^\prime\in\Omega_\omega$ satisfies 
$$
|\omega^\prime\cap\pts_i|\leq A |\pts_i\cap\omega|(D/D_i)^{-6+\eps}\leq AA_i|\pts| D_i^{-6}|(D/D_i)^{-6+\eps}\leq A|\pts|D^{-6+\eps};
$$
and 
$$
\sum_{\omega\in\Omega}\sum_{\omega^\prime\in\Omega_\omega}|\mathcal{Z}_{\omega^\prime}| \leq \sum_{\omega\in\Omega}A(D/D_i)^{4+\eps}|\mathcal{Z}_\omega| \leq A(D/D_i)^{4+\eps} A_iD_i^4|\mathcal{Z}|\leq A D^{4+\eps}|\mathcal{Z}|.
$$

We also obtain a set of incidences $I_{\omega^\prime}\subset I(\pts_{\omega^\prime},\mathcal{Z}_{\omega^\prime})$. We have
\begin{equation}\label{extraIncidences}
\begin{split}
&\sum_{\omega\in\Omega}\sum_{\omega^\prime\in\Omega_{\omega}}|I_{\omega^\prime}|\leq \sum_{\omega\in\Omega}\sum_{\omega^\prime\in\Omega_\omega}A(D/D_i)^{\eps}|\pts_{\omega^\prime}|\\
&\leq AD^{\eps}D_i^{-\eps}|\pts_i^\prime|\\
&\leq \frac{A}{2}D^{\eps}|\pts_i^\prime|.
\end{split}
\end{equation}

Define 
$$
J^\prime_i = J_i \cup \bigcup_{\omega\in\Omega}\bigcup_{\omega^\prime\in\Omega_\omega}I_{\omega^\prime}.
$$ 
We have
$$
|J^\prime_i|\leq AD^{\eps}|\pts_i^\prime|.
$$ 

Finally, applying the induction hypothesis gives us a set of points $\pts^*_{\omega^\prime}\subset\pts\cap\omega^\prime$. We have
\begin{equation}
\begin{split}
\sum_{\omega\in\Omega}\sum_{\omega^\prime\in\Omega_\omega}|\pts^*_{\omega^\prime}|&\leq\sum_{\omega\in\Omega}\sum_{\omega^\prime\in\Omega_\omega}A(D/D_i)^{6+\eps}\\
&\leq A A_i D_i^6(D/D_i)^{6+\eps}\\
&\leq \frac{A}{6} D^{6+\eps}.
\end{split}
\end{equation}

Define 
$$
\pts^*_i = \bigcup_{\omega\in\Omega_i}\bigcup_{\omega^\prime\in\Omega_\omega}\pts^*_{\omega^\prime}.
$$ 

Recall that we performed the above steps for each $i=0,\ldots,5$ with $\dim W_i\geq 1$. If instead $\dim(W_i)= 0$, then let $i_0=i$ and let $\pts_{i_0}^*=\pts_{i_0}$. Note that if $A$ is chosen sufficiently large (depending only on $D_1,\ldots,D_{i_0}$, which in turn depends only on $\eps$), then $|\pts^*_{i_0}|\leq A/6$.

We are now ready to combine the cells from each of the partitionings described above. Define 
$$
\Omega = \bigcup_{i=0}^{i_0}\bigcup_{\omega\in\Omega_i}\Omega_\omega.
$$
For each $\omega^\prime\in\Omega$, we have a set $\mathcal{Z}_{\omega^\prime}$ of double-spheres. Define

$$
\pts^* = \bigcup_{i=1}^{i_0} \pts^*_i.
$$

Let
$$
I_0 = J_0^\prime\cup\ldots\cup J_{i_0-1}^\prime.
$$

We have $|\Omega|\leq AD^{6+\eps}$,  $|I_0|\leq \sum_{i=1}^{i_0-1} AD^{\eps}|\pts_i^\prime|\leq AD^{\eps}|\pts|$, and $|\pts^*|\leq A D^{6+\eps}$. Thus the sets $\Omega,I_0,$ and $\{\mathcal{Z}_{\omega}\}_{\omega\in\Omega}$ satisfy the conclusions of the lemma. 
\end{proof}

\subsection{Proof of Lemma \ref{circlePartitioningLem}}
We now have the necessary tools to prove Lemma \ref{circlePartitioningLem}. For the reader's convenience, we will recall it here.

\begin{circlePartitioningLemLem}
Let $0\leq\alpha< 1/2$. For each $\eps>0$, there is a constant $A$ so that the following is true. Let $\pts\subset\RR^3$ be a set of $m$ points and let $\circles$ be a set of $n>4^{1/(1-2\alpha)}$ circles in $\RR^3$, no two of which are co-planar\footnote{see Remark \ref{coplanarRemark}}. 

Then for each $E\geq 1$, at least one of the following three things must hold:\\
\begin{itemize}
\item[(A)] (The circles are contained in a small number of spheres): There is a set $\mathcal{C}_{(A)}\subset \mathcal{C}$ of size at least $n/\log n$, and a set $\mathcal{S}$ of spheres with $|\mathcal{S}|\leq n^\alpha$. For each $S\in\mathcal{S}$, there is a set $\circles(S)\subset\circles_{(A)}$ of circles contained in $S$ so that the sets $\{\circles(S)\}_{S\in\mathcal{S}}$ are disjoint and $|\circles(S)|\leq 2n|\mathcal{S}|^{-1}$ for each $S\in\mathcal{S}$. Finally,
\begin{equation}\tag{\ref{incidencesSumOverSpheres}}
I(\pts,\circles_{(A)})=  \sum_{S\in\mathcal{S}}I(\pts, \circles(S)).
\end{equation}

\item[(B)] (The circles are contained in an intermediate number of spheres): There is a set $\mathcal{C}_{(B)}\subset \mathcal{C}$ of size at least 
\begin{equation}\tag{\ref{sizeOfCB}}
|\mathcal{C}_{(B)}|\geq E^{-\eps}n-E^{6+2\eps},
\end{equation}
a number $1\leq F<E$, and a set $\Omega$ of at most $AF^{6+\eps}$ ``cells.'' For each cell $\omega\in\Omega$, there is a set $\pts_\omega\subset\pts$ of points, a set $\circles_\omega\subset\circles_{(B)}$ of circles with $|\circles_\omega|\leq AF^{-6+\eps}n$, and a set $\mathcal{S}_\omega$ of spheres with $|\mathcal{S}_\omega|\leq |\circles_\omega|^{\alpha}$. For each $S\in\mathcal{S}_\omega$, there is a set $\circles(S)\subset\circles_\omega$ of circles contained in $S$.

The sets $\{\circles_\omega\}_{\omega\in\Omega}$ are disjoint and for each $\omega\in\Omega$, the sets $\{\circles(S)\}_{S\in\mathcal{S}_\omega}$ are disjoint. The sets $\{\pts_\omega\}_{\omega\in\Omega}$ are not disjoint, but 
$$
\sum_{\omega\in\Omega} |\pts_\omega|\leq A F^{4+\eps}m.
$$ 
For each $\omega\in\Omega$ and each $S\in \mathcal{S}_\omega,$ we have 
$
|\circles(S)|\leq A |\circles_\omega|^{1-\alpha}.
$
Finally,
\begin{equation}\tag{\ref{incidencesSumOverCellsSpheres}}
I(\pts,\circles_{(B)})\leq \sum_{\omega\in\Omega}\sum_{S\in\mathcal{S}_\omega}I(\pts_\omega,\circles(S)) + AE^{\eps}n.
\end{equation}

\item[(C)] (The circles strongly avoid being contained in spheres): There is a set $\mathcal{C}_{(C)}\subset \mathcal{C}$ of size at least $E^{-\eps}n-AE^{6+2\eps}$, and a set $\Omega$ of at most $AE^{6+\eps}$ ``cells.'' For each cell $\omega\in\Omega$ there is a set $\pts_\omega$ of points and a set $\circles_\omega$ of circles. The sets $\{\circles_\omega\}$ are disjoint, and $|\circles_\omega|\leq AE^{-6+\eps}n$ for each $\omega\in\Omega$. The sets $\{\pts_\omega\}$ are not disjoint, but $\sum_\Omega |\pts_\omega|\leq AE^{4+\eps}m$. For each $\omega\in \Omega$, at most $|\circles_{\omega}|^{1-\alpha}$ circles from $\circles_\omega$ can be contained in a common sphere. We also have

\begin{equation}\tag{\ref{incidencesSumNoSpheres}}
I(\pts,\circles)\leq \sum_{\omega\in\Omega}I(\pts_\omega,\circles_\omega)+ AE^{\eps}n.
\end{equation}
\end{itemize}
\end{circlePartitioningLemLem}

\begin{proof}
Let $\mathcal{S}_0$ be the set of spheres satisfying $|\circles_S|\geq |\circles|^{1-\alpha}$. Since $\alpha<1/2$ we can use the inclusion-exclusion principle to show that that $|\mathcal{S}_0|< 2n^\alpha$. Indeed, suppose there exists a set $\mathcal{S}_1\subset\mathcal{S}_0$ of cardinality $2n^{\alpha}$. Since each pair of spheres can contain at most one common circle, we have
\begin{equation*}
\begin{split}
2n & \leq \sum_{S\in\mathcal{S}_1}|\circles_S|\\
&\leq \Big|\bigcup_{S\in\mathcal{S}_1}\circles_S \Big|+\sum_{S_1,S_2\in\mathcal{S}_1}|\circles_{S_1}\cap\circles_{S_2}|\\
&\leq n + (2n^{\alpha})^2\\
&<2n,  
\end{split}
\end{equation*}
where on the last line we used the fact that $n> 4^{1/(1-2\alpha)}$. Since this is a contradiction, we conclude that $|\mathcal{S}_0|< 2n^\alpha$.

 We will first consider the case where 
 \begin{equation}\label{halfCirclesTrappedInZ0}
 \Big|\bigcup_{S\in\mathcal{S}_0}\circles_S\Big|\geq\frac{1}{4}|\circles|.
\end{equation}
Assign each circle in $\bigcup_{S\in\mathcal{S}_0}\circles_S$ uniquely to one set $\circles(S)\subset \circles_S$. After dyadically pigeonholing the varieties $S\in\mathcal{S}_0$, there is a number $t$ and a set of varieties $\mathcal{S}\subset\mathcal{S}_0$ so that $t\leq |\circles(S)|<2t$ for each $S\in\mathcal{S}$ and $\sum_{S\in\mathcal{S}}|\circles(S)|\geq n/\log n$. Let $\circles_{(A)}=\bigcup_{S\in\mathcal{S}}\circles(S)$. We conclude that Case (A) holds.

Now suppose that \eqref{halfCirclesTrappedInZ0} fails. Let 
\begin{equation}\label{defnCircles0}
\circles_0=\circles\backslash \bigcup_{S\in\mathcal{S}_0}\circles_S.
\end{equation}

Let $G$ be a large parameter to be determined later and let $L$ be the smallest integer with $G^L\geq E$. Define $\Omega_0=\{\RR^6\}$ and define $\circles_{\RR^6}=\circles_0,$ where $\circles_0$ is the set from \eqref{defnCircles0}. Define $\pts_{\RR^6}=\pts$. Thus we have defined sets $\circles_{\omega}$ and $\pts_\omega$ for each $\omega\in\Omega_0$.

Let $A_1$ be the constant from Lemma \ref{partitionCirclesLem} corresponding to $\eps$. Observe that if $i=0$ then
\begin{align}
&|\Omega_i|\leq A_1^i G^{(6+\eps)i}.\label{correctNumberCells}\\
&\textrm{The sets}\ \{\circles_\omega\}_{\omega\in\Omega}\ \textrm{are disjoint}.\\
&\sum_{\omega\in\Omega}|\circles_\omega|\geq 2^{-i-1}n-2^iA_1^i G^{(6+\eps)i}.\label{mostCirclesPreserved}\\
& A_1^{-i}G^{(-6+\eps)i}n \leq |\circles_\omega|\leq A_1^{i}G^{(-6+\eps)i}n\ \textrm{for each}\ \omega\in\Omega_i.\label{correctNumberCirclesPerCell}\\
&\sum_{\omega\in\Omega_i}|\pts_\omega|\leq A_1^i G^{(4+\eps)i}m.\label{correctNumberPtsPerCell}\\
&\textrm{For each}\ \omega\in\Omega\ \textrm{and for each sphere}\ S,\ |(\circles_{\omega})_S|\leq|\circles_\omega|^{1-\alpha}.\label{notTooManyCirclesInSphereInCell}\\
&\Big|\mathcal{I}(\pts,\circles_i)\backslash \bigcup_{\omega\in\Omega}\mathcal{I}(\pts_\omega,\circles_\omega)\Big|\leq i(A_1G^\eps)n\label{mostIncidencesCapturedPerCell}.
\end{align}

Now let $0\leq i \leq L$. We will describe the induction step. Suppose that $\Omega_i$ is a set of cells, and that for each $\omega\in\Omega_{i}$ there is a set of points $\pts_\omega$ and $\circles_\omega$ so that the conditions \eqref{correctNumberCells}--\eqref{mostIncidencesCapturedPerCell} hold.

For each $\omega\in\Omega_i$, let $\mathcal{Q}_\omega\subset\RR^6$ be the set of points associated to the circles in $\circles_\omega$ using the duality described in Section \ref{dualSpacePtsCirclesSec}. If $C$ is a circle, we will use $q_C\in\RR^6$ to denote the corresponding point. Let $\mathcal{Z}_{\omega}=\{Z_p\colon p\in\pts_{\omega}\}$. Use Lemma \ref{partitionCirclesLem} to partition $\mathcal{Q}_\omega$ into $\leq A G^{6+\eps}$ cells; call this set of cells $\tilde\Omega_\omega$. For each $\omega\in\Omega_{i}$ and each $\omega^\prime\in\tilde\Omega_\omega$, define 
\begin{equation*}
\begin{split}
\tilde\circles_{\omega^\prime}&=\{C\in\circles_{\omega}\colon q_C \in\omega^\prime\},\\
\pts_{\omega^\prime}&=\{p\in\pts_{\omega}\colon Z_p\in(\mathcal{Z}_{\omega})_{\omega^\prime}\},
\end{split}
\end{equation*}
where $(\mathcal{Z}_{\omega})_{\omega^\prime}=\{Z\in\mathcal{Z}_\omega\colon Z\cap\omega^\prime\neq\emptyset\}$.

From \eqref{sizeOfPStar} from Lemma \ref{partitionCirclesLem}, for each $\omega\in\Omega_i$, we have 
$$
\sum_{\omega^\prime\in\tilde\Omega_\omega}|\tilde\circles_{\omega^\prime}|\geq |\circles_{\omega}|-A_1 G^{6+\eps}.
$$

Let 
$$
\Omega_{\omega}=\{\omega^\prime\in\tilde\Omega_\omega\colon |\tilde \circles_{\omega^\prime}|\geq (4 A G^{6+\eps})^{-1}|\circles_{\Omega_\omega}|\}.
$$
Since each set $\Omega_\omega$ contains at most $AG^{6+\eps}$ cells, we have that

\begin{equation}\label{fewCirclesLostInSmallCell}
\sum_{\omega^\prime\in\Omega_\omega}|\tilde\circles_{\omega^\prime}|\geq \frac{3}{4}|\circles_{\omega}|-A_1 G^{6+\eps}.
\end{equation}

Define
$$
\Omega_{i+1}=\bigcup_{\omega\in\Omega_i}\Omega_\omega.
$$
We have that the sets $\{\circles_{\omega}\}_{\omega\in\Omega_{i+1}}$ are disjoint, and 
\begin{equation}\label{sumCirclesOverAllCells}
\begin{split}
&\sum_{\omega^\prime\in\Omega_{i+1}}|\tilde\circles_{\omega^\prime}|\\
&\geq \Big(\sum_{\omega\in\Omega}\frac{3}{4}|\circles_{\omega}|\Big)-(A_1 G^{6+\eps})(2^iA_1^iG^{(6+\eps)i}) \\
&\geq \frac{3}{4}\ 2^{-i-1}n-2^iA_1^i G^{(6+\eps)i}-2^iA_1^{i+1}G^{(6+\eps)(i+1)}\\
&\geq \frac{3}{4}\ 2^{-i-1}n-2^{i+1}A_1^{i+1} G^{(6+\eps)(i+1)}.
\end{split}
\end{equation}
For each $\omega\in\Omega_{i+1}$,
$$
|\tilde\circles_{\omega}|\leq (A_1 G^{-6+\eps})(A_1^i G^{(-6+\eps)i}n) \leq A_1^{i+1}G^{(6+\eps)(i+1)}n,
$$
and 
$$
\sum_{\omega\in\Omega_{i+1}}|\pts_{\omega}|\leq \sum_{\omega\in\Omega_i}\sum_{\omega^\prime\in\Omega_\omega}|\pts_{\omega^\prime}|\leq\sum_{\omega\in\Omega_i} A_1G^{4+\eps}|\pts_\omega|\leq (A_1G^{4+\eps})(A_1^i G^{(4+\eps)i}m)\leq A_1^{i+1}G^{(4+\eps)(i+1)}m.
$$

Thus \eqref{correctNumberCells}, \eqref{correctNumberCirclesPerCell}, and \eqref{correctNumberPtsPerCell} are satisfied. Furthermore by \eqref{boundOnI0} from Lemma \ref{partitionCirclesLem}  we have,

\begin{equation}\label{caseBMostIncidencesCaptured}
\begin{split}
&\Big| I(\pts,\circles_i)\backslash \bigcup_{\omega\in\Omega_{i+1}}I(\pts,\tilde\circles_\omega)\Big| \\
& \leq \Big| I(\pts,\circles_i)\backslash \bigcup_{\omega\in\Omega_{i}}I(\pts,\circles_\omega)\Big|+\sum_{\omega\in\Omega_i}\Big|I(\pts,\circles_\omega)\backslash \bigcup_{\omega^\prime\in\Omega_\omega} I(\pts,\tilde\circles_{\omega^\prime})\Big|\\
&\leq iAG^{\eps}n + \sum_{\omega\in\Omega_i}AG^{\eps}|\circles_\omega|\\
&\leq (i+1)AG^{\eps}n.
\end{split}
\end{equation}

For each $\omega\in\Omega_{i+1},$ let $\mathcal{S}_\omega$ be the set of spheres satisfying
\begin{equation}\label{defnOfZOmega}
|S_{\tilde\circles_\omega}|\geq2|\tilde\circles_{\omega}|^{1-\alpha}.
\end{equation}
(Recall that $S_{\tilde\circles_\omega}=\{C\in \tilde\circles_\omega\colon C\subset S\}$). Suppose that
\begin{equation}\label{mostCirclesTrappedInRichSpheres}
\sum_{\omega\in\Omega_{i+1}}\Big|\bigcup_{S\in \mathcal{S}_\omega}S_{\tilde\circles_\omega}\Big|\geq \frac{1}{4}\sum_{\omega\in\Omega_{i+1}}|\tilde\circles_\omega|.
\end{equation}
We are now in Case (B). Let $F=G^{i+1}$ and let $\Omega=\Omega_{i+1}$. For each $\omega\in\Omega$, let 
$$
\circles_\omega=\bigcup_{S\in\mathcal{S}_\omega}S_{\tilde\circles_\omega}.
$$

Let 
$$
\mathcal{C}_{(B)}=\bigcup_{\omega\in\Omega_{i+1}}\circles_\omega.
$$
We have that 
\begin{equation}
\begin{split}
|\mathcal{C}_{(B)}|&\geq\frac{1}{4}\sum_{\omega\in\Omega_{i+1}}|\circles_\omega|\\
&\geq 2^{-i-4}|\circles|-2^{i}A^{i} G^{(6+\eps)(i)}\\
&\geq n 2^{-L-4} - (2AG^{6+\eps})^{L}\\
&\geq nE^{-\eps}-E^{6+2\eps},
\end{split}
\end{equation}
provided $G$ is selected sufficiently large (depending on $A_1$ and $\eps$) so that 
$$
(2A_1)^L\leq E^{\log (2A_1) / \log G} \leq E^{\eps}. 
$$
Thus \eqref{sizeOfCB} holds.

For each $\omega\in\Omega$ and each $S\in\mathcal{S}_\omega$, select a set $\circles(S)\subset S_{\circles_\omega}$ so that the sets $\{\circles(S)\}$ are disjoint and their union is $\circles_\omega$. Using the inclusion-exclusion principle (the same argument used at the beginning of the proof of Lemma \ref{circlePartitioningLem}) shows that $|\mathcal{S}_\omega|\leq |\circles_\omega|^{\alpha}$. 

Note that 
\begin{equation}\label{circlesSHasCorrectSize}
|\circles(S)|\leq 4A_1G^{6+\eps} |\circles_{\omega}|^{1-\alpha}\leq A|\circles_{\omega}|^{1-\alpha},
\end{equation}
since the parent cell of $\omega$ satisfies \eqref{notTooManyCirclesInSphereInCell}, and $|\circles_{\omega}|\geq (4A_1G^{6+\eps})^{-1}|\circles_{\omega_p}|,$ where $\omega_p$ is the parent cell of $\omega$.

It remains to verify that  \eqref{incidencesSumOverCellsSpheres}  holds. But
\begin{equation*}
\begin{split}
\Big|&\mathcal{I}(\pts,\circles_{(B)})\backslash \bigcup_{\omega\in\Omega}\bigcup_{S\in\mathcal{S}_\omega}\mathcal{I}(\pts_\omega,\circles(S))\Big| \\
&\leq \Big|\mathcal{I}(\pts,\circles_i)\backslash \bigcup_{\omega\in\Omega}\mathcal{I}(\pts_\omega,\tilde\circles_\omega)\Big|\\
&\leq (i+1)(A_1G^\eps)n\\
&\leq \frac{\log E}{\log G}(A_1G^{\eps})n\\
&\leq AE^{\eps}n,
\end{split}
\end{equation*}
where on the third line we used \eqref{caseBMostIncidencesCaptured}. This concludes the analysis of Case (B). 

Suppose instead that \eqref{mostCirclesTrappedInRichSpheres} fails. This means we are not in Case (B) and we need to continue the induction process. For each $\omega\in\Omega_{i+1}$, let 
\begin{equation}\label{caseNotBDefnCirclesOmega}
\circles_{\omega}=\{C\in\omega\colon C\not\subset S\ \textrm{for any}\ S\in\mathcal{S}_\omega\}.
\end{equation}
By \eqref{sumCirclesOverAllCells} and the failure of \eqref{mostCirclesTrappedInRichSpheres}, we have 
\begin{equation*}
\begin{split}
\sum_{\omega\in\Omega}|\circles_\omega|&\geq\frac{3}{4}\sum_{\omega\in\Omega}|\tilde\circles_\omega|\\
&\geq \big(\frac{3}{4}\big)^2\ 2^{-i-1}n-2^{i}A_1^{i+1} G^{(6+\eps)(i+1)}\\
&\geq 2^{-(i+1)-1}n-2^{i}A_1^{i+1} G^{(6+\eps)(i+1)},
\end{split}
\end{equation*}
and thus \eqref{mostCirclesPreserved} holds. By \eqref{caseBMostIncidencesCaptured} we have that \eqref{mostIncidencesCapturedPerCell} holds. By the definition of $\circles_{\omega}$ from \eqref{caseNotBDefnCirclesOmega}, we have that \eqref{notTooManyCirclesInSphereInCell} also holds. 

We repeat this procedure for $i=1,\ldots,L$. If we are never in Case (B), then let $\Omega=\Omega_L$ and let $\circles_{(C)}=\circles_L$. The sets $\Omega,$ $\{\circles_\omega\}_{\omega\in\Omega}$ and $\{\pts_\omega\}_{\omega\in\Omega}$ satisfy the requirements of Case (C). 
\end{proof}

\section{Incidences between points and unit spheres}
In this section we will prove a bound on the number of incidences between $n$ points and $n$ unit spheres in $\RR^3$. Up to a multiplicative constant, this is equivalent to bounding the number of unit distances spanned by $n$ points in $\RR^3$.


Our bound will depend on three parameters that we get to choose, which we will call $\alpha,\ D,$ and $E$. In Section \ref{optimizeParamsSec}, we will determine the optimal values of these parameters. Intuitively, however, the reader should think of $\alpha$ as being close to $1/3$, $D$ as being close to $n^{1/4}$, and $E$ as being close to $n^{1/10}$. In particular, we will require that $0\leq\alpha <1/2$.

\begin{lem}
Let $\pts$ be a set of $m$ points and let $\spheres$ be a set of $n$ unit spheres in $\RR^3$. Let $1\leq D\leq n^{1/3}$ be an integer. Then there exists a set $\circles_0$ of $O(n^2D^{-3})$ circles so that at most $n$ circles lie on a common plane or (arbitrary) sphere, and
\begin{equation}
I(\pts,\spheres)\leq I(\mathcal{Q},\circles_0)+ O(D^2n) + O(D^4),
\end{equation}
where $\mathcal{Q}\subset\RR^3$ is the set of centers of the spheres in $\spheres.$
\end{lem}
\begin{proof}
Apply Theorem \ref{projectedPartitioning} (with $d=3,e=0,f=3$) with parameter $D$ to obtain a polynomial $P\in\RR[x_1,x_2,x_3]$ that partitions $\RR^3$ into $O(D^3)$ cells, each of which contains $O(nD^{-3})$ points from $\pts$. Call this set of cells $\Omega_0$.

Factor $Z(P)$ into irreducible components. Applying Lemma \ref{irreducibleIdealLem} if necessary, we can assume that each component generates a real ideal. For each irreducible component $P_i$, let $D_i=\deg(P_i)$, let $\pts_i=\pts\cap Z(P_i)\backslash Z(P)_{\operatorname{sing}}$, and let $m_i=|\pts_i|$. Let $\mathcal{I}_1$ be the set of indices $i$ satisfying $m_iD^3 \geq mD_i^3$ and let $\mathcal{I}_2$ be the remaining indices.

For each $i\in \mathcal{I}_1$, use Theorem \ref{partitioningOnAVariety} to select a partitioning polynomial $Q_i$ of degree 
$$
E_i=\Big\lfloor\frac{m_iD^3}{mD_i}\Big\rfloor^{1/2}.
$$
We have that $\pts_i$ is partitioned into $O(D_iE_i^2)=O(D^3 m_i/m)$ cells, each of which contains $O(mD^{-3})$ points. Call this set of cells $\Omega_i$. Note that 
$$
\sum_{i\in\mathcal{I}_1}D_iE_i\leq D^{3/2}m^{-1/2}\sum_i D_i^{1/2}m_i^{1/2}\leq D^{3/2}m^{-1/2}\Big(\sum_i D_i\Big)^{1/2}\Big(\sum_i m_i\Big)^{1/2} \leq D^2.
$$ 

For each $i\in \mathcal{I}_2$, use Theorem \ref{projectedPartitioning} (with $d=3,e=0,f=3$) to find a partitioning polynomial $Q_i$ of degree 
$$
E_i = \big\lfloor D(m_i/m)^{1/3}\big\rfloor.
$$
Note that $i\in\mathcal{I}_2$ implies that $E_i<D_i$, and thus the polynomial $Q_i$ does not vanish identically on $Z(P_i)$. Note that if $m_i < m/D^3$ then $E_i=0$. This means that the polynomial $Q_i$ is constant (i.e. $Q_i=1$ will suffice), so $Z(Q_i)=\emptyset$. 

If $E_i\geq 1$, we have that $\pts_i$ is partitioned into $O(E_i^3)$ cells, with $O(m_i/E_i^3)=O(m/D^3)$ points per cell. Similarly, if $E_i=0$, then $\pts_i$ is ``partitioned'' into one cell, which contains $O(m_i)=O(m/D^3)$ points. 

Call this set of cells $\Omega_i$. Note that 
$$
\sum_{i\in\mathcal{I}_2} E_i^2 \leq \sum_i D_i^2 \leq D^2.
$$ 

Let $\Omega=\Omega_0\cup\bigcup_{i\in\mathcal{I}_1}\Omega_i\cup\bigcup_{i\in\mathcal{I}_2}\Omega_i$. We have $|\Omega|=O(D^3)$;  $|\pts\cap\omega|=O(D^{-3}n)$ for each $\omega\in\Omega$;
\begin{equation}\label{sphereCellInteractions}
\sum_{\omega\in\Omega}|\mathcal{S}_\omega|\leq  CD^2 n
\end{equation}
and 
$$
I(\pts,\mathcal{S})\leq \sum_{\omega\in\Omega}I(\pts\cap\omega, \mathcal{S}_\omega)+O(D^2n) + O(D^4).
$$

From \eqref{sphereCellInteractions}, we have that
\begin{equation}\label{spheresToCircles}
I(\pts,\mathcal{S})\leq \sum_{\omega\in\Omega}|\{(p,q,S)\colon S\in\mathcal{S}, p,q\in \pts\cap\omega\cap S\} |+O(D^2n) + O(D^4).
\end{equation}

For each cell $\omega\in\Omega$ and each pair of points $p,q\in\omega$ with $|p-q|<1$, define $C_{p,q}$ to be the circle $S_p\cap S_q$. Define $\circles_0 = \bigcup_{\omega}\{C_{p,q}\colon p,q\in\omega\}$ and define  $\mathcal{Q}=\{q\colon S_q\in\spheres\}.$ We have

\begin{equation}\label{circlePointIncidences}
\sum_{\omega}|\{(p,q,S)\colon p,q\in\omega\cap S\}| \leq I(\mathcal{Q},\circles_0)+O(n^2 D^{-3}),
\end{equation}
where the $O(n^2 D^{-3})$ term account for pairs $(p,q,S)$ with $p,q\in \pts\cap\omega\cap S$ and $|p-q|=1$; for each of the $O(D^3)$ cells $\omega$, there are $O(n^2D^{-6})$ pairs $p,q\in\omega$ with $|p-q|=1$. Each of these pairs can contribute at most two triples $(p,q,S)$ to the above sum. 

Finally, note that for each plane or (arbitray) sphere $Z$ and each $p\in\pts$, there is at most one additional point $q\in\pts$ with $S_p\cap Z = S_q\cap Z$. Thus each plane or (arbitrary) sphere can contain at most $n$ circles from $\circles_0$.
\end{proof}

We will now continue with the proof of Theorem \ref{mainThm}. After applying an inversion around a suitably chosen sphere, we can assume without loss of generality that no pair of circles in $\mathcal{C}_0$ are coplanar (if two circles are co-planar before the inversion is applied, they will be co-spherical afterward. This is merely a notational convenience so we can refer to ``spheres,'' rather than ``planes and spheres''). Use dyadic pigeonholing to find a set $\circles\subset\circles_0$ and a number $t$ so that $t\leq |C\cap\mathcal{Q}|\leq 2t$ for each $C\in\circles$ and 
$$
I(\mathcal{Q},\circles_0) \leq (\log n) I(\mathcal{Q},\circles).
$$ 

Our goal is now to bound $I(\mathcal{Q},\circles)$. Note that since each circle in $\circles$ is incident to roughly the same number of points (up to a factor of two), it suffices to bound $I(\mathcal{Q},\circles^\prime)$, where $\circles^\prime\subset\circles$ is a set of circles with cardinality comparable to that of $\circles$. This observation will be used several times in the arguments that follow.

Define $N=|\circles| = O(n^2D^{-3})$. Let $\eps>0$. Apply Lemma \ref{circlePartitioningLem} to $\mathcal{Q}$ and $\circles$ with parameters $0\leq\alpha\leq 1/2,\ \eps,$ and $1\leq E \leq A^{-1}n^{1/6-2\eps}$. At least one of Cases (A), (B), or (C) must occur. We will deal with each of these situations in turn.

\subsection{Case (A)}\label{CaseASec}
Let $\circles_{(A)}$, $\mathcal{Z}$, and $\{\circles(S)\}_{S\in\mathcal{Z}}$ be the output from Lemma \ref{circlePartitioningLem}. After diadic pigeonholing, we can find a set $\mathcal{Q}^\prime\subset\mathcal{Q}$ so that 
\begin{equation}\label{QPrimeBoundsQ}
I(\mathcal{Q},\mathcal{C})\lessapprox_{\eps}I(\mathcal{Q}^\prime,\mathcal{C}), 
\end{equation}
and each point in $\mathcal{Q}^\prime$ is incident to roughly the same number of circles from $\mathcal{C}_{(A)}$ (up to a factor of two). Note that it will (likely) no longer be the case that each circle from $\mathcal{C}_{(A)}$ is incident to the same number of points, but this will not be a problem. After further dyadic pigeonholing, we can assume that each sphere in $\mathcal{Z}$ contains roughly the same number of points from $\mathcal{Q}^\prime$ (up to a factor of two).

Let $\mathcal{Z}_1$ be the set of spheres in $\mathcal{Z}$ that are $1/2$-degenerate (in the sense of Definition \ref{defnNonDegenerate}) with respect to the point set $\mathcal{Q}^\prime$, and let $\mathcal{Z}_2$ be the remaining spheres. Let $\circles_{(A),1})\subset\circles_{(A)}$ be the set of circles contained in spheres from $\mathcal{Z}_1$, and let $\circles_{(A),2})\subset\circles_{(A)}$ be the set of circles contained in spheres from $\mathcal{Z}_2$

For each $S\in\mathcal{Z}_1$, let $C_S$ be a circle that contains at least half the points from $S$. $C_S$ need not be an element of $\circles(S)$. Since each point in $\mathcal{Q}^\prime$ is incident to roughly the same number of circles, for each sphere $S\in\mathcal{Z}_1$ it suffices to bound the number of incidences between the circles from $\circles(S)$ and the points in $\mathcal{Q}^\prime\cap C_S$. Each circle $C\in\circles(S)$ with $C\neq C_S$ can be incident to at most two points from $\mathcal{Q}^\prime \cap C_S$. If $C_S\in\circles(S)$, then this circle can contribute an additional $n$ points. Using \eqref{QPrimeBoundsQ}, we have
\begin{equation}\label{contribDegenerate}
I(\mathcal{Q}, \circles_{(A),1})\leq |\mathcal{Z}|n\leq n^{1+\alpha}.
\end{equation}

We will now bound the number of incidences between the points in $\mathcal{Q}^\prime$ and the circles in $\circles_{(A),2}.$

\subsubsection{Case (A.1): $|\mathcal{Z}_2|\leq n^{1/4}$} 
By Theorem \ref{richSphereThm} applied to the set $\mathcal{Q}$, each of the spheres in $\mathcal{Z}_2$ contains $O(n|\mathcal{Z}|^{-1})$ points from $\mathcal{Q}^\prime$. We also have that each sphere contains $\leq n$ circles from $\circles$. 

Applying the bound from Theorem \ref{pointCircleBound} to the $1/2$-non-degenerate spheres, we obtain
\begin{equation}\label{caseA1}
\begin{split}
I(\mathcal{Q},\circles_{(A),2})&\lessapprox |\mathcal{Z}|\Big(\big(n/|\mathcal{Z}|\big)^{6/11}n^{9/11}+\big(n/|\mathcal{Z}|\big)^{2/3}n^{2/3}+\big(n/|\mathcal{Z}|\big)+n\Big)\\
&\lessapprox |\mathcal{Z}|^{5/11}n^{15/11}+|\mathcal{Z}|^{1/3}n^{4/3}+|\mathcal{Z}|n\\
&\lessapprox n^{(1/4)(5/11) + 15/11}+n^{1/4\cdot 1/3+4/3}+n^{1/4+1}\\
&\lessapprox n^{65/44}+n^{17/12}+n^{5/4}\sim n^{1.47727}.
\end{split}
\end{equation}

\subsubsection{Case (A.2): $n^{1/4}\leq |\mathcal{Z}_2|\leq n^{\alpha}$} By Theorem \ref{richSphereThm}, each of the spheres in $\mathcal{Z}_2$ contains $O(n^{4/5}|\mathcal{Z}_2|^{-1/5})$ points from $\mathcal{Q}$, and $\lessapprox_\eps N/|\mathcal{Z}_2| \lessapprox_\eps n^2D^{-3}/|\mathcal{Z}_2|$ circles. We apply Theorem \ref{pointCircleBound} and obtain
\begin{equation}\label{caseA2}
\begin{split}
I(\mathcal{Q},\circles_{(A),2})&\lessapprox_\eps|\mathcal{Z}_2|\Big(\big(n^{4/5}|\mathcal{Z}_2|^{-1/5}\big)^{6/11}\big(n^2D^{-3}/|\mathcal{Z}_2|\big)^{9/11}+\big(n^{4/5}|\mathcal{Z}_2|^{-1/5}\big)^{2/3}\big(n^2D^{-3}/|\mathcal{Z}_2|\big)^{2/3}\\
&\phantom{|\mathcal{Z}_2|\Big(}\ \ \  +\big(n^{4/5}|\mathcal{Z}_2|^{-1/5}\big)+\big(n^2D^{-3}/|\mathcal{Z}_2|\big)\Big)\\
&\leq |\mathcal{Z}_2|^{4/55}n^{114/55}D^{-27/11}+|\mathcal{Z}_2|^{1/5}n^{28/15}D^{-2}+ n^{4/5}|\mathcal{Z}_2|^{4/5} + n^2D^{-3}\\
&\leq n^{114/55+(4/55)\alpha}D^{-27/11}+n^{28/15+(1/5)\alpha}D^{-2}+ n^{4/5+(4/5)\alpha} + n^2D^{-3}.
\end{split}
\end{equation}

\subsection{Case (B)}\label{caseBSection}
Let $\circles_{(B)}$, $F$, $\Omega$, $\{\mathcal{Q}_\omega\}_{\omega\in\Omega}$, $\{\mathcal{S}_\omega\}_{\omega\in\Omega}$, and $\{ \circles(S)\}_{\omega\in\Omega,\ S\in \mathcal{S}_\omega}$ be the output from Lemma \ref{circlePartitioningLem}. First, we can assume that $|\mathcal{Q}_\omega|\lessapprox_{\eps} M_1 = nF^{-2}$ for each cell $\omega\in\Omega$. This is because $\sum_{\omega\in\Omega}|\mathcal{Q}_\omega|\lessapprox_{\eps}nF^4$, so we can discard the cells for which $|\mathcal{Q}_\omega|\geq n^{A_1\eps} n F^{-2}$. If $A_1$ is selected sufficiently large, then at most $AF^{-6-A_1\eps/2}$ of the cells from $\Omega$ are discarded. Since each discarded cell contains at most $AF^{-6+\eps}n$ circles from $\circles_{(B)}$, at most half the circles from $\circles_{(B)}$ lie in cells that are discarded. Since each circle is incident to roughly the same number of points (up to a factor of two), this process only affects the number of incidences by a multiplicative factor of four. After further pigeonholing, inside each cell, we can select a set $\mathcal{Q}_\omega^\prime\subset\mathcal{Q}_\omega$ so that each point in $\mathcal{Q}_\omega^\prime$ is incident to roughly the same number of circles from $\circles_\omega$. Again, after this process has been performed it is likely that the circles will no longer each be incident to the same number of points, but this will not be a problem for us.

Next, after dyadic pigeonholing the (non-discarded) cells in $\Omega$ and the circles $C\in\circles_\omega$, we can ensure that each cell in $\Omega$ contains $N_1 \lessapprox_{\eps} n^2D^{-3}F^{-6}$ circles (up to a multiplicative factor of two); these circles are evenly distributed (up to a multiplicative factor of two) over $K \lessapprox_\eps (n^2D^{-3}F^{-6})^{\alpha}=n^{2\alpha}D^{-3\alpha}F^{-6\alpha}$ disjoint spheres, each of which contains roughly (again, up to a factor of two) the same number of points. We perform this pigeonholing so that the total number of incidences is reduced by (at most) a factor of $\lessapprox_{\eps}1$. Of course it is still the case that each of these spheres contains $\lessapprox_\eps n^{2(1-\alpha)}D^{-3(1-\alpha)}F^{-6(1-\alpha)}$ circles, since each sphere contains at most $AN_2^{1-\alpha}$ circles, where $N_2$ is the number of circles originally present in the cell. 

Let $\mathcal{S}_{\omega,1}\subset\mathcal{S}_{\omega}$ be the set of spheres that are $1/2$-degenerate (in the sense of Definition \ref{defnNonDegenerate}) with respect to the point set $\mathcal{Q}_\omega^\prime,$ and let $\mathcal{S}_{\omega,2}$ be the remaining spheres. Let $\circles_{(B),1}$ be the circles contained in spheres from $\mathcal{S}_{\omega,1}$ (as $\omega$ ranges over the cells in $\Omega$), and let $\circles_{(B),2}$ be the remaining circles.

For each $S\in\mathcal{S}_1$, let $C_S\subset S$ be a circle that contains at least half the points of $\mathcal{Q}_{\omega}^\prime\cap S$; it suffices to bound the number of incidences between circles in $\circles(S)$ and points in $\mathcal{Q}_{\omega}^\prime\cap C_S$.  Since each circle in $\circles(S)$ that is distinct from $C_S$ can intersect $C_S$ in at most two places, we have 
$$
\sum_{\omega\in\Omega}\sum_{S\in\mathcal{S}_{\omega,1}}I\big(\circles(S)\backslash\{C_S\},\ \mathcal{Q}_{\omega}^\prime\cap C_S\big)\lessapprox F^6K(N_1/K)\lessapprox n^2D^{-3}.
$$

It remains to bound the incidences between circles of the form $\{C_S\colon S\in\mathcal{S}_{w,1}\}$ and the points in $\mathcal{Q}_\omega^\prime$. We have 
\begin{equation*}
\begin{split}
&\sum_{\omega\in\Omega}I\Big(\bigcup_{S\in\mathcal{S}_{\omega,1}}\{C_S\},\ \mathcal{Q}_{\omega^\prime}\Big)\\
&\qquad\lessapprox F^6(M_1^{6/11}K^{9/11}+M_1^{2/3}K^{2/3}+M_1+K)\\
&\qquad\lessapprox
F^6\Big( (nF^{-2})^{6/11}(n^{2\alpha}D^{-3\alpha}F^{-6\alpha})^{9/11}+(nF^{-2})^{2/3}(n^{2\alpha}D^{-3\alpha}F^{-6\alpha})^{2/3}+(nF^{-2})+(n^{2\alpha}D^{-3\alpha}F^{-6\alpha})\Big)\\
&\qquad\lessapprox E^{54/11 - (54/11) \alpha}n^{6/11 + (18/11)\alpha}D^{(-27/11)\alpha}+E^{14/3 - 4 \alpha}n^{2/3+(4/3)\alpha}D^{-2\alpha}+nE^4+E^{6-6\alpha}n^{2\alpha}D^{-3\alpha}.
\end{split}
\end{equation*}

Thus
\begin{equation}\label{contribDegenerateB}
\begin{split}
I(\mathcal{Q}, \mathcal{C}_{(B),1})&\lessapprox n^2D^{-3}+E^{54/11 - (54/11) \alpha}n^{6/11 + (18/11)\alpha}D^{(-27/11)\alpha}\\
&+E^{14/3 - 4 \alpha}n^{2/3+(4/3)\alpha}D^{-2\alpha}+nE^4+E^{6-6\alpha}n^{2\alpha}D^{-3\alpha}.
\end{split}
\end{equation}

We will now bound incidences between points and circles in $\circles_{(B),2}$.

By Theorem \ref{richSphereThm}, for each cell $\omega\in\Omega$, each sphere in $\mathcal{S}_{\omega,2}$ contains 
\begin{equation}\label{numberPtsPerSphereCaseB}
\lessapprox_\eps \max\Big(M_1/K,\ M_1^{4/5}K^{-1/5}\Big)
\end{equation}
points from $\mathcal{Q}$.

\subsection{Case (B.1): First term of \eqref{numberPtsPerSphereCaseB} dominates}
The total number of incidences is
\begin{equation}\label{caseB1}
\begin{split}
I(\mathcal{Q},\circles_{(B),2})&\lessapprox_\eps F^6 K \Big( (n F^{-2}K^{-1})^{6/11}
(n^{2(1-\alpha)}D^{-3(1-\alpha)}F^{-6(1-\alpha)})^{9/11}\\
&\qquad\qquad\qquad\qquad\qquad+(nF^{-2}K^{-1})^{2/3}(n^{2(1-\alpha)}D^{-3(1-\alpha)}F^{-6(1-\alpha)})^{2/3}\\
&\qquad\qquad\qquad\qquad\qquad+(nF^{-2}K^{-1}) +(n^{2(1-\alpha)}D^{-3(1-\alpha)}F^{-6(1-\alpha)})\Big)\\
%
%
&\lessapprox_\eps E^{(24/11)\alpha}n^{24/11-(8/11)\alpha}D^{-27/11+(12/11)\alpha}\\
&\qquad+E^{2/3 + 2\alpha}n^{2-(2/3)\alpha}D^{-2+\alpha}+E^4n+n^{2}D^{-3}.
\end{split}
\end{equation}

\subsection{Case (B.2): Second term of \eqref{numberPtsPerSphereCaseB} dominates}
The total number of incidences is
\begin{equation}\label{caseB2}
\begin{split}
I&(\mathcal{Q},\circles_{(B),2})\\
&\lessapprox_\eps F^6 K \Big( (n^{4/5}F^{-8/5}K^{-1/5})^{6/11}
(n^{2(1-\alpha)}D^{-3(1-\alpha)}F^{-6(1-\alpha)})^{9/11}\\
&\qquad\qquad\qquad\qquad\qquad+(n^{4/5}F^{-8/5}K^{-1/5})^{2/3}(n^{2(1-\alpha)}D^{-3(1-\alpha)}F^{-6(1-\alpha)})^{2/3}\\
&\qquad\qquad\qquad\qquad\qquad+(n^{4/5}F^{-8/5}K^{-1/5}) +(n^{2(1-\alpha)}D^{-3(1-\alpha)}F^{-6(1-\alpha)})\Big)\\
&\leq F^{12/55 - (24/55)\alpha}n^{114/55 + (8/55)\alpha}D^{-27/11 - (12/55)\alpha}+F^{14/15 - (6/5)\alpha}n^{28/15 + (2/5)\alpha}D^{-2 - (3/5)\alpha}\\
&\ \ +F^{14/5 - (24/5)\alpha}n^{4/5 + (8/5)\alpha}D^{-(12/5)\alpha}+n^2D^{-3}\\
&\leq E^{12/55 - (24/55)\alpha}n^{114/55 + (8/55)\alpha}D^{-27/11 - (12/55)\alpha}+E^{14/15 - (6/5)\alpha}n^{28/15 + (2/5)\alpha}D^{-2 - (3/5)\alpha}\\
&\ \ +E^{22/5 - (24/5)\alpha}n^{4/5 + (8/5)\alpha}D^{-(12/5)\alpha}+n^2D^{-3}.
\end{split}
\end{equation}
\subsection{ Case (C)} 
Let $\Omega,$ $\{\mathcal{Q}_\omega\}_{\omega\in\Omega}$, and $\{\circles_{\omega}\}_{\omega\in\Omega}$ be the output from Lemma \ref{circlePartitioningLem}. Recall that for each $\omega\in\Omega,$ we have $|\circles_\omega|\lessapprox_{\eps}n^2D^{-3}E^{-6}$. An identical argument to that in Section \ref{caseBSection} allows us to assume that for each $\omega\in\Omega$, we have $|\mathcal{Q}_\omega|\lessapprox_{\eps}nE^{-2}$. Furthermore, for each $\omega\in\Omega$, we have that at most $\lessapprox_{\eps}\Big(n^2D^{-3}E^{-6}\Big)^{1-\alpha}=n^{2-2\alpha}D^{-3+3\alpha}E^{-6+6\alpha}$ circles from $\circles_{\omega}$ lie in a common plane or sphere. We apply the cutting from Corollary \ref{cuttingCirclesR3} to the circles in each cell; this gives us 
$$
\lessapprox_{\eps}\Big(n^2D^{-3}E^{-6}\Big)^{4/3} + \Big(n^2D^{-3}E^{-6}\Big)^{3/2-\alpha/2} = n^{8/3}D^{-4}E^{-8}+n^{3-\alpha}D^{-9/2+3\alpha/2}E^{-9+3\alpha}
$$ 
pseudo-segments,  and apply Lemma \ref{pointPseudosegmentIncidencesLem} to the resulting set of pseudo-segments. We conclude that the total number of incidences is at most 
\begin{equation}\label{caseC}
\begin{split}
I&(\mathcal{Q},\circles_{(C)})\\
&\lessapprox_\eps E^6\Big( (nE^{-2})^{1/2}(n^2D^{-3}E^{-6})^{3/4}+(nE^{-2})^{2/3}(n^2D^{-3}E^{-6})^{1/3}(n^{2-2\alpha}D^{-3+3\alpha}E^{-6+6\alpha})^{1/3}\\
&\quad\qquad+ (nE^{-2}) + n^{8/3}D^{-4}E^{-8}+n^{3-\alpha}D^{-9/2+2\alpha/2}E^{-9+3\alpha} \Big)\\
&\lessapprox_\eps E^{1/2}n^{2}D^{-9/4}  +   E^{2/3 + 2\alpha}D^{-2+\alpha}n^{2 - (2/3)\alpha}+E^4n\\
&\qquad\qquad\qquad\qquad+n^{8/3}D^{-4}E^{-2} + n^{3-\alpha}D^{-9/2+3\alpha/2}E^{-3+3\alpha}.
\end{split}
\end{equation}

\subsection{Optimizing the parameters}\label{optimizeParamsSec}
We must now select values of $\alpha,D$, and $E$ that minimize the number of point-circle incidences that come from \eqref{spheresToCircles}, \eqref{caseA1}, \eqref{caseA2}, \eqref{caseB1}, \eqref{caseB2}, and \eqref{caseC}. We can phrase this as a convex optimization problem---define $\beta$ and $\delta$ so that $D= n^\beta,\ E = n^\delta$. Our goal is now to minimize the maximum of the following terms
\begin{center}
\begin{tabular}[h!]{|c|c|}
\hline
From \eqref{spheresToCircles} \& \eqref{circlePointIncidences} : & From \eqref{contribDegenerate}:\\
\\[-1em]
$1+2\beta$ & $1+\alpha$\\
\\[-1em]
$2-3\beta$ & \\
\\[-1em]
\hline
From \eqref{caseA1}: & From \eqref{caseA2}:\\ 
\\[-1em]
$\frac{65}{44}$ & 
$\frac{114}{55}+\frac{4}{55}\alpha - \frac{27}{11}\beta$\\
\\[-1em]
& $\frac{28}{15}+\frac{1}{5}\alpha - 2\beta$  \\
\\[-1em]
&$\frac{4}{5}+\frac{4}{5}\alpha$\\
\\[-1em]
&$2-3\beta$\\
\hline
$\qquad\qquad\qquad\qquad$from \eqref{contribDegenerateB}:$\qquad\qquad\qquad\qquad$ & $\qquad\qquad\ \ \ $from \eqref{caseB1}:$\qquad\qquad\ \ \ $\\  %
\\[-1em]
$2 - 3\beta$ & 
$\frac{24}{11}\alpha\delta + \frac{24}{11}-\frac{8}{11}\alpha -\frac{27}{11}\beta+\frac{12}{11}\alpha\beta$\\
\\[-1em]
$6/11 + (18/11)\alpha - (27/11)\alpha\beta+(54/11)\delta - (54/11) \alpha\delta$ & 
$\frac{2}{3}\delta + 2\alpha\delta + 2 -\frac{2}{3}\alpha -2 \beta +\alpha\beta$\\
\\[-1em]
$2/3+(4/3)\alpha - 2\alpha\beta+(14/3)\delta - 4 \alpha\delta$ 
& $1+4\delta$\\
\\[-1em]
$1+4\delta$ & 
$2-3\beta$\\
\\[-1em]
$2\alpha -3\alpha\beta+6\delta-6\alpha\delta$\\

\hline
From \eqref{caseB2}: & From \eqref{caseC}:\\
\\[-1em]
$\frac{12}{55}\delta - \frac{24}{55}\alpha\delta+\frac{114}{55} + \frac{8}{55}\alpha -\frac{27}{11}\beta - \frac{12}{55}\alpha\beta$ &
$\frac{1}{2}\delta+2-\frac{9}{4}\beta$\\
\\[-1em]
$\frac{14}{15}\delta - \frac{6}{5}\alpha\delta+  \frac{28}{15} + \frac{2}{5}\alpha -2\beta - \frac{3}{5}\alpha\beta$ & 
$\frac{2}{3}\delta + 2\alpha\delta -2\beta+\alpha\beta +2 - \frac{2}{3}\alpha$\\
\\[-1em]
 $\frac{22}{5}\delta - \frac{24}{5}\alpha\delta + \frac{4}{5} + \frac{8}{5}\alpha -\frac{12}{5}\alpha\beta$&
$4\delta+1$\\
\\[-1em]
$2 -3\beta$ &
$\frac{8}{3} -4\beta -2\delta$\\
\\[-1em]
 &
$3-\alpha -\frac{9}{2}\beta+\frac{3}{2}\alpha\beta -3\delta+3\alpha\delta$\\
\hline
\end{tabular}
\end{center}
subject to the constraints 
\begin{equation}\label{constraints}
0\leq\alpha\leq1/2\quad\textrm{and}\quad0\leq\delta\leq 1/3-\beta/2. 
\end{equation}
We can verify that 
\begin{equation}
\begin{split}
\alpha&= 5/17,\\
\beta&= 49/197,\ \textrm{and}\\
\delta&= 37/394
\end{split}
\end{equation}
satisfies the constraints from \eqref{constraints} (these values of $\alpha$, $\beta$ and $\delta$ were found using convex optimization). When these values of $\alpha,\beta,$ and $\delta$ are inserted into the equations from the above table, we conclude that
$$
I(\mathcal{Q},\circles)\lessapprox_\eps n^{295/197}.
$$

This concludes the proof of Theorem \ref{mainThm}.

\begin{rem}
The following terms from the above bounds achieve the maximum bound of $n^{295/197}$. The terms are in the format ``(equation): bound.''
\begin{itemize}
\item[\eqref{spheresToCircles}:] $nD^2$.
\item[\eqref{caseB1}:] $E^{(24/11)\alpha}n^{24/11-(8/11)\alpha}D^{-27/11+(12/11)\alpha}$  and $E^{2/3 + 2\alpha}n^{2-(2/3)\alpha}D^{-2+\alpha}$.
\item[\eqref{caseB2}:] $E^{12/55 - (24/55)\alpha}n^{114/55 + (8/55)\alpha}D^{-27/11 - (12/55)\alpha}$ and $E^{14/15 - (6/5)\alpha}n^{28/15 + (2/5)\alpha}D^{-2 - (3/5)\alpha}$.
\item[\eqref{caseC}:]$E^{2/3 + 2\alpha}D^{-2+\alpha}n^{2 - (2/3)\alpha}$ and $n^{3-\alpha}D^{-9/2+3\alpha/2}E^{-3+3\alpha}$.
\end{itemize} 
\end{rem}

\begin{rem}
The above proof could be generalized to bound the number of incidences between $m$ points and $n$ unit spheres with $m\neq n$. However, the number of terms to be considered in the analogues of Equations \eqref{spheresToCircles}, \eqref{caseA1}, \eqref{caseA2}, \eqref{caseB1}, \eqref{caseB2}, and \eqref{caseC} would increase substantially.
\end{rem}

\section{Point-circle incidences in $\RR^3$}
Implicit in the proof of Theorem \ref{mainThm} are several new incidence bounds for points and circles in $\RR^3$. We will describe one of these bounds in greater detail. It will be convenient to introduce the following notation. We say $F = O^*(G)$ if $F = O_{\eps}(n^{\eps}G)$ for every $\eps>0$. 

\begin{thm}
Let $\pts$ be a set of $m$ points and let $\circles$ be a set of $n$ circles in $\RR^3$. Suppose that at most $q$ circles are contained in a common plane or sphere. Then the number of point-circle incidences is
\begin{equation}\label{pointCircleIncidenceBd}
O^*\big(m^{1/2}n^{3/4}+m^{2/3}n^{13/15}+m^{1/3}n^{8/9} + nq^{2/3} + m\big).
\end{equation}
\end{thm}
If $q$ is small, then this is superior to the previous best-known bound of
$$
O^*\big(m^{3/7}n^{6/7}+m^{2/3}n^{1/2}q^{1/6}+m^{6/11}n^{15/22}q^{3/22}+m+n\big).
$$
from \cite{SSZ2}.

\begin{proof}
After dyadic pigeonholing, we can assume that each circle in $\mathcal{C}$ is incident to approximately the same number of points from $\pts$. Apply Lemma \ref{circlePartitioningLem} with $\alpha = 1/2+\eps$ and 
\begin{equation}\label{valueOfE}
E=\min\Big(n^{1/6}q^{-1/3},\ n^{7/30}m^{-1/5},\  n^{2/9}m^{-1/6},\ 1\Big).
\end{equation}
Since $E\leq n^{1/6}q^{-1/3}$, Case (C) must occur. Thus we obtain a set $\Omega$ of $O^*(E^6)$ cells, with $O^*(nE^{-6})$ circles in each cell, and $\sum_{\omega\in\Omega}|\pts_\omega| = O^*(E^4m).$ 

Note that $q= O\big((nE^{-6})^{1/2}\big) = O\big((nE^{-6})^{2/3}\big)$. Thus we can apply Lemma \ref{cuttingCirclesR3} to cut the circles in each cell into $O\big((nE^{-6})^{4/3+\eps}\big) = O(n^{4/3+\eps}E^{-8+6\eps})$ pseudo-segments. Apply Lemma \ref{pointPseudosegmentIncidencesLem} to the points and pseudo-segments inside each cell. We conclude that the number of incidences is
\begin{equation}\label{incidenceBoundCirclesSec}
\begin{split}
I(\pts,\circles)&\lessapprox_\eps E^6 \Big( (mE^{-2})^{1/2}(nE^{-6})^{3/4}+mE^{-2}+n^{4/3}E^{-8}\Big)\\
&\lessapprox_\eps E^{1/2}m^{1/2}n^{3/4} + E^4m + E^{-2}n^{4/3}.
\end{split}
\end{equation}
Combining \eqref{incidenceBoundCirclesSec} and \eqref{valueOfE}, we obtain \eqref{pointCircleIncidenceBd}.
\end{proof}



\section{Acknowledgments}
The author would like to thank Jozsef Solymosi for numerous helpful discussions and the anonymous referee for corrections and suggestions.

\end{document}